\documentclass[11pt]{article}
\usepackage[english]{babel}
\usepackage{amsmath,amsthm}
\usepackage{amsfonts}
\usepackage{mathrsfs}
\usepackage{color}
\usepackage{bbm}
\usepackage{amsrefs}
\newtheorem{theorem}{Theorem}[section]

\newtheorem{lemma}[theorem]{Lemma}

\newtheorem{assumption}[theorem]{Assumption}
\theoremstyle{definition}
\newtheorem{definition}[theorem]{Definition}
\theoremstyle{remark}
\newtheorem{remark}[theorem]{Remark}

\numberwithin{equation}{section}
\begin{document}
\def\Pro{{\mathbb{P}}}
\def\E{{\mathbb{E}}}
\def\e{{\varepsilon}}
\def\ds{{\displaystyle}}
\def\nat{{\mathbb{N}}}
\def\Dom{{\textnormal{Dom}}}
\def\dist{{\textnormal{dist}}}
\def\H{{\mathcal{H}}}
\def\I{{\mathcal{I}}}
\def\Imu{{\mathcal{I}_\mu}}
\def\Amu{{\mathcal{A}_\mu}}
\def\Smu{{\mathcal{S}_\mu}}
\def\Smustar{{\mathcal{S}_\mu^\star}}

\title{Smoluchowski-Kramers approximation for the damped stochastic wave equation with multiplicative noise in any spatial dimension}%
\author{M. Salins \footnote{
Boston University, Department of Mathematics and Statistics, 111 Cummington Mall, Boston, MA, 02215}}
%\thanks{}%
%\date{}%
% ----------------------------------------------------------------
\maketitle

\begin{abstract}
  We show that the solutions to the damped stochastic wave equation converge pathwise to the solution of a stochastic heat equation. This is called the Smoluchowski-Kramers approximation. Cerrai and Freidlin have previously demonstrated that this result holds in the cases where the system is exposed to additive noise in any spatial dimension or when the system is exposed to multiplicative noise and the spatial dimension is one. The current paper proves that the Smoluchowski-Kramers approximation is valid in any spatial dimension when the system is exposed to multiplicative noise.
\end{abstract}

\section{Introduction} \label{S:intro}
The motion of an elastic material in a region $D \subset \mathbb{R}^d$ exposed to friction as well as deterministic and random forcing can be described by the damped stochastic wave equation
\begin{equation} \label{eq:intro-wave}
  \begin{cases}
  \mu \frac{\partial^2 u^\mu}{\partial t^2}(t,x) = \Delta u^\mu(t,x) - \frac{\partial u^\mu}{\partial t}(t,x)+ b(t,x,u^\mu(t,x)) \\
  \hspace{4.5cm}+g(t,x,u^\mu(t,x)) Q \frac{\partial w}{\partial t } (t,x),\\
  u^\mu(t,x) = 0 , \ \ \ x \in \partial D \\
  u^\mu(0,x) = u_0(x), \ \ \ \frac{\partial u^\mu}{\partial t}(0,x) = v_0(x).
  \end{cases}
\end{equation}
In the above equation, $\mu>0$ is the mass-density of the material. The forcing term $\Delta u$ describes the forces neighboring particles exert on each other, $-\partial u/\partial t$ models a constant friction term, $b$ is a nonlinear forcing term, and $g Q \partial w/ \partial t$ is a space and time dependent stochastic forcing. The noise is driven by $w(t)$, a $L^2(D)$-cylindrical Wiener processes \cite[Chapter 4.2.1]{dpz}. The Dirichlet boundary conditions guarantee that the boundary of the elastic material is fixed. Initial conditions are also prescribed.

We study the asymptotics of the solutions to this equation as the mass density $\mu \to 0$ and demonstrate that the solutions converge to the solutions of a stochastic heat equation
\begin{equation} \label{eq:intro-heat}
  \begin{cases}
    \frac{\partial u}{\partial t}(t,x) = \Delta u(t,x) +b(t,x,u(t,x)) + g(t,x,u(t,x))Q \frac{\partial w}{\partial t}(t,x),\\
    u(t,0) = 0, \ \ \ x \in \partial D,  \ \ \ u(0,x) = u_0(x).
  \end{cases}
\end{equation}
The heat equation can be thought of as \eqref{eq:intro-wave} with $\mu$ formally replaced by $0$

This limit, the Smoluchowski-Kramers approximation, was first investigated by Smoluchowski \cite{s-1916} and Kramers \cite{k-1940} for finite dimensional diffusions of the form
\begin{equation}
  \mu \ddot{X}^\mu(t) = b(t,X^\mu(t)) -\dot{X}^\mu(t) + g(t,X^\mu(t))\dot{W}(t)
\end{equation}
where  $X^\mu$ is $\mathbb{R}^d$-valued, $b: [0,+\infty) \times \mathbb{R}^d \to \mathbb{R}^d$ is a vector field and $g: [0,+\infty) \times \mathbb{R}^d \to \mathbb{R}^{d\times k}$, and $W(t)$ is a $k$-dimensional Wiener process. As $\mu \to 0$ the solutions converge pathwise on finite time intervals to the solution of the first-order equation
\begin{equation}
  \dot{X}(t) = b(t,X(t)) + g(t,X(t))\dot{W}(t).
\end{equation}
Furthermore, the first-order equation approximates some longer-time behaviors including invariant measures and exit time problems. Many Smoluchowski-Kramers results for finite dimensional systems are summarized in \cite{f-2004} including pathwise convergence, invariant measures, Wong-Zakai approximation, homogenization, and large deviations. Various generalizations including the presence of state-dependent friction have been investigated in the finite dimensional case \cite{fh-2011,fhw-2013,cf-2005,hmvw-2015,hhv-2016,l-2014,bhvw-2017,cf-2011,s-2007,cf-2015,hs-2017}.

The Smoluchowski-Kramers approximation for stochastic partial differential equations such as \eqref{eq:intro-wave} were first investigated by Cerrai and Freidlin \cite{cf-2006-add,cf-2006-mult}. In \cite{cf-2006-add}, they considered the additive noise case where $g(t,x,u) \equiv 1$ and in \cite{cf-2006-mult}, they considered the multiplicative noise case when the spatial dimension $d=1$. In each case they show that the solutions $u^\mu(t,x)$ of \eqref{eq:intro-wave} converge to the solutions of \eqref{eq:intro-heat} pathwise in probability, in the sense that for any $T>0$ and $\delta>0$
\begin{equation} \label{eq:conv-in-prob}
  \lim_{ \mu \to 0} \Pro \left(\sup_{t \in [0,T]}\int_D |u^\mu(t,x)-u(t,x)|^2 dx>\delta \right)=0.
\end{equation}
The Smoluchowski-Kramers approximation in the presence of a magnetic field and Smoluchowski-Kramer's interplay with large deviations in the small noise regime for infinite dimensional systems have also been investigated \cite{cs-2014-smolu-grad,cs-2016-mag,cs-2016-smolu-AoP,lr-2014,cdz-2013}.

The main results of this paper fill a gap in the literature by demonstrating that the Smoluchowski-Kramers approximation is valid in the case of multiplicative noise in any spatial dimension $d\geq 1$ under the assumptions that the noise covariance $Q$ satisfies appropriate assumptions. Furthermore, the methods in this paper allow us to improve from convergence in probability as in \eqref{eq:conv-in-prob} to $L^p$ convergence. In particular the main result of this paper, Theorem \ref{thm:convergence-u}, proves that for any $T>0$ and $p\geq1$,
\begin{equation} \label{eq:Lp-Smolu}
  \lim_{\mu \to 0} \E \sup_{t \in [0,T]}\left( \int_D |u^\mu(t,x) - u(t,x)|^2 dx \right)^{p/2} = 0.
\end{equation}

If $D \subset \mathbb{R}^d$ is an open region with smooth boundary then there is a complete orthonormal basis of $L^2(D)$ consisting of eigenfunctions of $\Delta$ such that $\Delta e_k(x) = -\alpha_k e_k(x)$ for an increasing sequence of eigenvalues $\alpha_k\geq0$. Weyl's Theorem \cite[page 356]{evans} guarantees that the eigenvalues of $-\Delta$ with Dirichlet boundary conditions behave like $\alpha_k \sim k^{2/d}$ as $k \to +\infty$. In dimension $d=1$,  the eigenvalues have the useful property that $\sum_{k=1}^\infty \frac{1}{\alpha_k}< +\infty$. A consequence is that \eqref{eq:intro-wave} is well-defined when is exposed to white noise (the case where $Q=I$ is the identity) (see \cite{cf-2006-mult}). In dimensions $d\geq 2$, the noise must be more regular than white noise in order for \eqref{eq:intro-wave} to be well-defined.

In the additive noise case considered in \cite{cf-2006-add}, the Smoluchowski-Kramers approximation is proved under the assumption that $Q$ is diagonalized by the same basis of eigenfunctions as the Laplacian with eigenvalues $Q e_k = \lambda_k e_k$ and that $\sum_{k=1}^\infty \frac{\lambda_k^2}{\alpha_k^{1-\theta}} < +\infty$ for some $\theta \in (0,1)$. This is also the minimal condition that guarantees that the solutions to \eqref{eq:intro-wave} and \eqref{eq:intro-heat} are well-defined and function valued.

The minimal conditions on the noise covariance $Q$ that guarantee that the heat equation with multiplicative noise \eqref{eq:intro-heat} is well-defined and function valued are characterized in \cite{c-2003}. In the dimension $d=1$ case, \eqref{eq:intro-heat} is well-defined if the eigenvalues of $Q$ are assumed to be uniformly bounded. In dimensions $d\geq 2$, \eqref{eq:intro-heat} is well-defined the eigenvalues of $Q$ are assumed to satisfy
\begin{equation} \label{eq:intro-Q}
  \sum_{j=1}^\infty \lambda_j^q<+\infty \text{ for some $2<q< \frac{2d}{d-2}$.}
\end{equation}
 In the $d=2$ case, this means that $2<q<+\infty$.

In this paper, we show that the solutions to \eqref{eq:intro-wave} exist and are function valued under the same conditions on the eigenvalues of $Q$. This requires a novel proof because the argument of \cite{c-2003} relied on the fact that the heat equation semigroup is analytic, but the wave equation semigroup is not analytic. Furthermore, we show that the Smoluchowski-Kramers approximation is valid in the sense that \eqref{eq:Lp-Smolu} holds under these same minimal assumptions on $Q$.

The proofs of the well-posedness of \eqref{eq:intro-wave} and the Smoluchowski-Kramers approximation \eqref{eq:Lp-Smolu} are both based on a careful analysis of the wave equation semigroup.

The paper is organized as follows. In Section \ref{S:notation}, we describe the assumptions and notations used in the paper. In Section \ref{S:heat}, we recall some results about the heat equation. In Section \ref{S:main}, we state the main results of this paper.  In Section \ref{S:semigroup-estimates}, we carefully analyze the properties of the wave equation semigroup. In Section \ref{S:stoch-conv}, we analyze the properties of the stochastic convolutions with the wave equation semigroup. In Section \ref{S:well-posed}, we apply the results from Sections \ref{S:semigroup-estimates} and \ref{S:stoch-conv} to prove that the stochastic wave equation is well-defined. Finally, in Section \ref{S:convergence} we prove that the mild solutions to the stochastic wave equation converge to the mild solution of the stochastic heat equation.

\section{Assumptions and notations} \label{S:notation}

We consider the damped stochastic wave equation  \eqref{eq:intro-wave} under the following assumptions.
\begin{assumption} \label{assum:bg}
  The functions $b: [0,+\infty)\times D\times \mathbb{R}\to \mathbb{R}$ and $g: [0,+\infty)\times D \times \mathbb{R} \to \mathbb{R}$ are uniformly Lipschitz continuous and have sublinear growth in the third variable. There exists $C\geq 0$ such that for any $u,v \in \mathbb{R}$,
  \begin{equation} \label{eq:Lipschitz}
    \sup_{\substack{x \in D\\ t\geq0}} \left(|b(t,x,u)-b(t,x,v)| + |g(t,x,u) - g(t,x,v)| \right) \leq C|u-v|.
  \end{equation}
  and
  \begin{equation} \label{eq:linear-growth}
    \sup_{\substack{x \in D\\ t\geq0}} \left(|b(t,x,u)| + |g(t,x,u)| \right) \leq C(1 + |u|).
  \end{equation}
\end{assumption}

Define $H=L^2(D)$ and let $A$ be the realization of the Laplace operator in $H$ with Dirichlet boundary conditions. There exists a sequence of eigenfunctions of $A$ that form a complete orthonormal basis of $H$. We list the eigenvalues in increasing order $0< \alpha_1\leq \alpha_k\leq \alpha_{k+1}$ so that
\[Ae_k = -\alpha_k e_k.\]
\begin{assumption} \label{assum:A}
  Assume that the domain $D \subset \mathbb{R}^d$ is regular enough so that
  \begin{equation} \label{eq:weyl}
    \alpha_k \sim k^{2/d}
  \end{equation}
  and
  \begin{equation} \label{eq:bounded-eigens}
    \sup_k|e_k|_{L^\infty(D)}< +\infty.
  \end{equation}
\end{assumption}
Assumption \ref{assum:A} holds, for example, when $D$ is a generalized rectangle in $\mathbb{R}^d$.

The cylindrical Wiener process $w(t)$ is defined as the formal sum
\begin{equation}
  w(t) = \sum_{k=1}^\infty e_k \beta_k(t)
\end{equation}
where $\{\beta_k(t)\}$ is a sequence of independent one-dimensional Brownian motion on a common probability space. Integration against a cylindrical Wiener process is defined in \cite[Chapter 4.2.1]{dpz}.
\begin{assumption} \label{assum:Q}
  The operator $Q \in \mathscr{L}_+(H)$  is diagonilized by the same orthonormal basis of $H$  as $A$. $Q$ has eigenvalues $\lambda_j\geq 0$ satisfying
  \[Q e_j = \lambda_j e_j.\]
  If $d=1$, then $Q$ is a bounded operator in the sense that
  \begin{equation} \label{eq:Q-bounded}
    \sup_{j \in \mathbb{N}} \lambda_j < +\infty.
  \end{equation}
  If $d\geq 2$, there exists $2<q< \frac{2d}{d-2}$ such that
  \begin{equation} \label{eq:lambda-sum}
    \sum_{j=1}^\infty \lambda_j^{q} < +\infty.
  \end{equation}
\end{assumption}

\begin{remark}
  The condition that $q< \frac{2d}{d-2}$ guarantees that $\frac{q}{q-2}>\frac{d}{2}$. Therefore, by Assumption \ref{assum:A},
  \begin{equation} \label{eq:alpha-sum}
    \sum_{k=1}^\infty \alpha_k^{-q/(q-2)} < +\infty.
  \end{equation}
\end{remark}
\begin{remark}
  Assumption \ref{assum:Q} is the same as Assumption 2 in \cite{c-2003} for the stochastic reaction diffusion equation. This means that the same conditions that imply the well-posedness of the reaction diffusion equation are sufficient for the analysis of this paper.
\end{remark}

For $\delta \in \mathbb{R}$, define the Hilbert spaces $H^\delta$ to be the completion of $C_0^\infty(D)$ under the norm
\[|f|_{H^\delta}^2 = \sum_{k=1}^\infty \alpha_k^\delta\left<f,e_k\right>_H^2.\]
For $\delta>0$, these spaces are equivalent to the fractional Sobolev spaces $W^{\delta,2}_0(D)$ \cite{sobolev}.

It is helpful to study the wave equation as a pair in an appropriate phase space,
\begin{equation} \label{eq:pair}
  \begin{cases}
   \frac{\partial u}{\partial t}(t,x) = v(t,x),\\
   \frac{\partial v}{\partial t}(t,x) = \frac{1}{\mu}\left(\Delta u(t,x) - v(t,x) +b(t,x,u(t,x)) + g(t,x,u(t,x))Q \frac{\partial w}{\partial t}\right).
  \end{cases}
\end{equation}
Define the phase spaces $\H_\delta := H^\delta \times H^{\delta -1}$. We also use the notation $\H := \H_0$.
Define the linear operator $A_\mu: D(A_\mu)=\H_{\delta-1} \to \H_\delta$ by
\begin{equation} \label{eq:Amu-def}
  \Amu(u,v) = (v, Au/\mu - u/\mu).
\end{equation}
The operator $\Amu$ generates a $C_0$ semigroup $\Smu(t): \H_\delta \to \H_\delta$.

Define the composition mapping $B:[0,+\infty)\times H \to H$ by, for any $t\geq 0$ and $u \in H$
\begin{equation} \label{eq:B-def}
  B(t, u)(x) = b(t,x,u(x)).
\end{equation}
Define the composition operator $G: [0,+\infty) \times H \to \mathscr{L}(L^\infty(D):H)$ by, for any $t\geq0$, $u \in H$, and $h \in L^\infty(D)$,
\begin{equation} \label{eq:G-def}
  [G(t,u)h](x) = g(t,x,u(t,x))h(x).
\end{equation}
Note that for $u \in H$, $G(t,u)$ is also well-defined as a bounded linear mapping from $H$ to $L^1(D)$ by H\"older inequality.
Because of Assumption \ref{assum:bg}, $B$ and $G$ are Lipschitz continuous in the second variable.

%Consider the pair $z^\mu(t,x) = \left(u^\mu(t,x), \frac{\partial u^\mu}{\partial t}(t,x)\right)$ in phase space where $u^\mu$ solves \eqref{eq:intro-wave}.
Define $\Pi_1: \H_\delta \to H^\delta$ is the projection onto the first component and $\Pi_2: \H_\delta \to H^{\delta -1}$ is the projection onto the second component. That is, for any $(u,v) \in \H^\delta$,
\begin{equation} \label{eq:Pi-defs}
  \Pi_1 (u,v) =u, \text{ and } \Pi_2 (u,v) = v.
\end{equation}
Define $\I_\mu: H^\delta \to \H_\delta$ such that
\begin{equation} \label{eq:Imu-def}
  \I_\mu u = (0,u/\mu).
\end{equation}
The equation \eqref{eq:pair} can be rewritten in the abstract formulation where $z^\mu(t) = (u^\mu(t),v^\mu(t))$
\begin{equation} \label{eq:wave-abstract}
  dz^\mu(t) = \Amu z^\mu(t) + \I_\mu B(t,\Pi_1z^\mu(t)) + \I_\mu G(t,\Pi_1z^\mu(t))Qdw(t).
\end{equation}

\begin{definition}
The mild solution to \eqref{eq:wave-abstract} is defined to be the solution of the integral equation.
\begin{align} \label{eq:wave-mild}
  z^\mu(t) = &\Smu(t)z_0 + \int_0^t \Smu(t-s)B(s,\Pi_1z^\mu(s))ds \nonumber \\
  &+ \int_0^t \Smu(t-s)\I_\mu G(s,\Pi_1z^\mu(s))Qdw(s)
\end{align}
where $z_0 = (u_0,v_0)$. Then $u^\mu(t) = \Pi_1 z^\mu(t)$ is the mild solution to \eqref{eq:intro-wave}.
\end{definition}

For any $T>0$ the function spaces $C([0,T]:H)$ and $C([0,T]:\H)$ are the Banach spaces of $H$ (resp. $\H$)- values continuous functions on $[0,T]$. They are endowed with the supremeum norm
\begin{equation}
  |\varphi|_{C([0,T]:H)} : = \sup_{ t\in [0,T]} |\varphi(t)|_H, \ \ \ \ |\psi|_{C([0,T]:\H)} := \sup_{t \in [0,T]} |\psi(t)|_H.
\end{equation}

Let $(\Omega, \mathcal{F}, \Pro)$ be a probability space. For any Banach space $E$ the space $L^p(\Omega:E)$ is the set of all $E$-valued random variables with the property that $\E |\varphi|_E^p<+\infty$. $L^p(\Omega:E)$ is a Banach space. In this paper we are most interested in the case where $E = C([0,T]:H)$ or $E= C([0,T]:\H)$.

Throughout this paper, the letter $C$ refers to an arbitrary positive constant whose value can change from line to line.

\section{Heat Equation} \label{S:heat}
In this section we recall some of the well-posedness results for the heat equation \eqref{eq:intro-heat}. Using the notation of Section \ref{S:notation}, \eqref{eq:intro-heat} can be written in the abstract formulation in $H$
\begin{equation}
  du(t) = [Au(t) + B(t,u(t))]dt + G(t,u(t))Qdw(t).
\end{equation}
The mild solution for the heat equation is the solution to the integral equation
\begin{equation} \label{eq:heat-mild}
  u(t) = S(t)u_0 + \int_0^t S(t-s)B(s,u(s))ds + \int_0^t S(t-s)G(s,u(s))Qdw(s)
\end{equation}
where $S(t)$ is the heat equation semigroup, which satisfies $S(t)e_k = e^{-\alpha_k t} e_k$. All of the results of this section can be found in \cite{c-2003}.

Denote the heat equation's stochastic convolution by
\begin{equation} \label{eq:Gamma-heat-def}
  \Gamma(t) = \int_0^t S(t-s)\Phi(s)Qdw(s)
\end{equation}
where we will set $\Phi(t) = G(t,\varphi(s))$ or $\Phi(t) = (G(t,\varphi(t)) - G(t,\psi(t)))$.

By the factorization formula of \cite{dpz},
\[\Gamma(t) = \frac{\sin(\pi\alpha)}{\pi} \int_0^t (t-s)^{\alpha-1}S(t-s)\Gamma_\alpha(s)ds\]
where
\begin{equation} \label{eq:Gamma-alpha}
  \Gamma_\alpha(t) = \int_0^t (t-s)^{-\alpha} S(t-s)\Phi(s)Qdw(s).
\end{equation}

We collect some results that we will use later in the paper.
\begin{lemma} \label{lem:heat-Gamma-alpha-bound}
  For any $\alpha \in (0,1/2)$ satisfying $-2\alpha - \frac{d(q-2)}{2q}>-1$, $p > \frac{1}{\alpha}$, and any $T>0$, there exists $C=C(T,p,\alpha)>0$ such that for any $t \in [0,T]$,
  \begin{equation} \label{eq:heat-Gamma-bound}
    \E \left|\Gamma_\alpha(t) \right|_H^p \leq C \sup_{s \in [0,t]} \|\Phi(s)\|_{\mathscr{L}(L^\infty(D),H)}^p.
  \end{equation}
\end{lemma}
For more information about the proof of this Lemma see Lemma 3.3 of \cite{c-2003}.

\begin{lemma} \label{lem:heat-stoch-conv-bound}
  For $\alpha \in 0,1/2)$ satisfying $-2\alpha - \frac{d(q-2)}{2q}>-1$ and $p \geq \frac{1}{\alpha}$,
  \[\E \sup_{t \leq T} |\Gamma(t)|^p \leq CT \E \sup_{t \in [0, T]} \|\Phi(t)\|_{\mathscr{L}(L^\infty(D),H)}^p.\]
\end{lemma}

\begin{lemma} \label{lem:Gamma-tails}
  Let $P_N$ be the projection onto $\textnormal{span}\{e_k\}_{k=1}^N$. Let $\Phi$ fixed progressively measurable $\mathscr{L}(L^\infty(D),H)$ valued process satisfying
  \[\E \sup_{t \in [0,T]}\|\Phi(t)\|^p_{\mathscr{L}(L^\infty(D),H)} < +\infty. \]
  Then for any fixed $\alpha>0$ satisfying the conditions of Lemma \ref{lem:heat-stoch-conv-bound},
  \[\lim_{N \to +\infty} \E | (I-P_N) \Gamma_\alpha(t)|_H^p = 0.\]
\end{lemma}
\begin{proof}
  This is an immediate consequence of the dominated convergence theorem.
\end{proof}

The following Theorem is presented in \cite[Proposition 3.2]{c-2003} and we state it without proof.
\begin{theorem}[Proposition 3.2 of \cite{c-2003}] \label{thm:heat-well-posed}
  Assume that Assumptions \ref{assum:bg}, \ref{assum:A}, and \ref{assum:Q} hold. For any initial condition $u_0 \in H$, there exists a unique solution  $u\in L^p(\Omega: C([0,T]:H))$ to \eqref{eq:heat-mild}  where $p\geq 2$ satisfies the conditions of Lemma \ref{lem:heat-stoch-conv-bound}.
\end{theorem}
The proof is based on the well-posedness of the stochastic convolutions and a fixed point argument.

\section{Main results} \label{S:main}

The first main result of this paper is that the mild solutions $z^\mu$ solving \eqref{eq:wave-mild} are well defined.
\begin{theorem} \label{thm:wave-eq-well-posed}
  Assume that Assumptions \ref{assum:bg}, \ref{assum:A}, and \ref{assum:Q} hold. For any initial conditions $(u_0,v_0) \in \H$ and $\mu>0$, there exists a unique solution $z^\mu \in L^p(\Omega: C([0,T]:\H))$ to \eqref{eq:wave-mild}.
\end{theorem}
The proof of Theorem \ref{thm:wave-eq-well-posed} is given in Section \ref{S:well-posed}. The proof requires careful analysis of the Fourier decomposition of the wave equation semigroup and the stochastic convolution, which can be found in Sections \ref{S:semigroup-estimates} and \ref{S:stoch-conv}.

The next main result is that the Smoluchowski-Kramers approximation is valid for these wave equations with multiplicative noise in any spatial dimension. The convergence of $u^\mu$ to $u$ is in $L^p(\Omega:C([0,T]:H))$, which is an improvement over previous results, which were known to converge in probability. Furthermore, this result is true in any spatial dimension $d\geq 1$.
\begin{theorem}[Smoluchowski-Kramers approximation] \label{thm:convergence-u}
 Assume that Assumptions \ref{assum:bg}, \ref{assum:A}, and \ref{assum:Q} hold. Let $u$ be the mild solution of \eqref{eq:heat-mild} with initial condition $u_0 \in H$ and $u^\mu$ be the solution of \eqref{eq:wave-mild} with the same initial position $u_0$ and initial velocity $v_0 \in H^{-1}$. Under Assumptions \ref{assum:bg}, \ref{assum:A},  and \ref{assum:Q}, there exists $p \geq 2$ such that for any $T>0$,
 \begin{equation} \label{eq:conv-u}
   \lim_{ \mu \to 0} \E \sup_{t \in [0,T]}\left|u(t) - u^\mu(t) \right|_H^p = 0.
 \end{equation}
\end{theorem}
The proof of Theorem \ref{thm:convergence-u} is presented in Section \ref{S:convergence}.

\section{Estimates on the wave equation semigroup $\Smu(t)$} \label{S:semigroup-estimates}
In this section we investigate the properties of the semigroup $\Smu(t)$. The exact form of the semigroup can be found in \cite[Proposition 2.2]{cf-2006-add}. We briefly recall some of the main observations about this semigroup and then we introduce some new analysis. Because $A$ is diagonalized by the orthonormal basis $\{e_k\}$, for any $k \in \mathbb{N}$ the operator $\Amu$ is invariant on the two dimensional linear span in $\H$ of the form $\{(u_k e_k, v_k e_k): u_k,v_k \in \mathbb{R}\}$ . The semigroup $\Smu(t)$ is also invariant on each of these two-dimensional spans.

Let $u\in H$ and $v \in H^{-1}$. Set $u_k = \left<u,e_k\right>_H$, $v_k = \left<v,e_k\right>_H$, and let
\[f_k^\mu(t;u_k,v_k) = \left<e_k,\Pi_1 \Smu(t) (u_k e_k,v_k e_k) \right>_{H}\]
and
\[g_k^\mu(t;u_k,v_k) = \left<e_k, \Pi_2 \Smu(t) (u_k e_k,v_k e_k) \right>_H.\]
Then
\begin{equation} \label{eq:semigroup-decomp}
  \Smu(t)(u,v) = \sum_{k=1}^\infty \begin{pmatrix} f_k^\mu(t;u_k,v_k) e_k, & g_k^\mu(t;u_k,v_k) e_k \end{pmatrix}.
\end{equation}
By the definition of $\Amu$,   $g_k^\mu(t;u_k,v_k) = (f_k^\mu)'(t;u_k,v_k)$ and $f_k^\mu(t,u_k,v_k)$ solves
\begin{equation} \label{eq:f_k-def}
  \mu (f^\mu_k)''(t) + (f^\mu_k)'(t) + \alpha_k f(t) = 0, \ \ f^\mu_k(0) = u_k, \ \ (f^\mu_k)'(0) = v_k.
\end{equation}

To study the stochastic convolution, we will be particularly interested in the case where $u_k=0$ and $v_k = 1$. According to \cite[Proposition 2.2]{cf-2006-add},
 \begin{align} \label{eq:f_k-explicit}
   f^\mu_k(t;0,1) = \frac{\mu }{\sqrt{1-4\mu\alpha_k}} \Bigg[ &\exp\left(-t\left(\frac{1-\sqrt{1-4\mu\alpha_k}}{2\mu} \right)\right) \nonumber\\
   & - \exp \left(-t \left(\frac{1+\sqrt{1-4\mu\alpha_k}}{2\mu} \right) \right)  \Bigg].
 \end{align}
 We use the notation that when $1-4\mu\alpha_k <0$, $\sqrt{1-4\mu\alpha_k} := i \sqrt{4\mu\alpha_k -1}$. When $1-4\mu\alpha_k = 0$, $f_k^\mu(t;0,1):= t e^{-\frac{t}{2\mu}}$.
 We see that that the solutions to \eqref{eq:f_k-explicit}  feature different behaviors depending on whether $1-4\mu\alpha_k \geq 0$ or $1-4\mu\alpha_k <0$. When $1-4\mu\alpha_k\geq 0$, the behavior is dominated by the exponential term $\exp\left(-t\left(\frac{1-\sqrt{1-4\mu\alpha_k}}{2\mu} \right)\right)$. This exponent is bounded by $-\alpha_kt$ because
 \[-\frac{1-\sqrt{1-4\mu\alpha_k}}{2\mu} = -\frac{4\mu\alpha_k}{2\mu \left(1 + \sqrt{1-4\mu\alpha_k} \right)} \leq -\alpha_k.\]
 Consequently, for any fixed $\mu>0$  there are a finite number of $k \in \mathbb{N}$ satisfying $1-4\mu\alpha_k\geq 0$, and for this finite number of Fourier modes, $f_k^\mu(t;0,1)$ can be bounded by terms that behave like $\mu e^{-\alpha_k t}$.

 On the other hand, for the infinite number of modes satisfying $1-4\mu\alpha_k<0$,
 \begin{equation}
   f_k^\mu(t;0,1) = \frac{2\mu }{\sqrt{4\mu\alpha_k -1}} \exp\left( -\frac{t}{2\mu}\right)\sin \left(\frac{t\sqrt{4\mu\alpha_k-1}}{2\mu} \right).
 \end{equation}
 In this regime, the functions no longer behave like their parabolic analogue. They behave approximately as $\sqrt{\frac{\mu}{\alpha_k}}\exp\left( -\frac{t}{2\mu}\right)$ as $t\to +\infty$.
 These observations are verified in the next sequence of lemmas.

 \begin{lemma} \label{lem:f-bounds}
   Assume that $f_k^\mu(t;u,v)$ solves \eqref{eq:f_k-def}.
   \begin{enumerate}
     \item If $u=0$ and $1-4\mu\alpha_k \geq 0$, then
     \begin{equation}\label{eq:f-bound-alpha}
       |f_k^\mu(t;0,v)| \leq 4 \mu |v| e^{-\alpha_k t}
     \end{equation}
     and
     \begin{equation} \label{eq:f'-bound-alpha}
       |(f_k^\mu)'(t;0,v)| \leq 2|v|e^{-\alpha_k t}.
     \end{equation}
     \item If $u=0$ and $1-4\mu\alpha_k \leq 0$, then
     \begin{equation} \label{eq:f-bound-mu}
       |f_k^\mu(t;0,v)| \leq \frac{\sqrt{4\mu}|v|}{\sqrt{\alpha_k}} e^{-\frac{t}{4\mu}}
     \end{equation}
     and
     \begin{equation} \label{eq:f'-bound-mu}
       |(f_k^\mu)'(t;0,v)| \leq 2 |v| e^{-\frac{t}{4\mu}}
     \end{equation}
     \item For any $k \in \mathbb{N}$, $\mu>0$ and $u,v \in \mathbb{R}$,
      \begin{equation} \label{eq:f-f'-bounds}
        \mu|f_k^\mu(t;u,v)|^2 + \alpha_k |(f_k^\mu)'(t;u,v)|^2 \leq \mu|v|^2  + \alpha_k |u|^2.
      \end{equation}
   \end{enumerate}
 \end{lemma}

\begin{remark}
  An immediate consequence of \eqref{eq:f-f'-bounds} is that if $v =0$ and $u \in \mathbb{R}$, then for any $k \in \mathbb{N}$,
     \begin{equation} \label{eq:f-bound-sup}
       |f_k^\mu(t;u,0)| \leq |u|.
     \end{equation}
\end{remark}

 \begin{proof}%[Proof of Lemma \ref{lem:f-bounds} ]
   For the simplicity of notation, we let $f(t) = f_k^\mu(t;u,v)$ and specify $k$, $\mu$, $u$, and $v$ throughout the proof.
   Let $\gamma\geq0$ and define $h(t) = e^{\gamma t} f(t)$. We will set $\gamma$ to be either $\alpha_k$ or $\frac{1}{4\mu}$ depending on the relationship between $\alpha_k$ and $\mu$. $h$ solves the equation
   \begin{equation} \label{eq:h-eq}
     \begin{cases}
       \mu h''(t) + (1-2\mu\gamma)h'(t) + (\mu\gamma^2 - \gamma +\alpha_k)h(t) = 0,
       \\  h(0) = u, \ \ h'(0) = \gamma u + v.
     \end{cases}
   \end{equation}

   We calculate two energy estimates. First, by multiplying \eqref{eq:h-eq} by $h'(t)$,
   \[\frac{\mu}{2} \frac{d}{dt} |h'(t)|^2 + (1-2\mu \gamma) |h'(t)|^2 + \frac{1}{2}(\mu \gamma^2 -\gamma + \alpha_k)\frac{d}{dt} |h(t)|^2=0.\]
   Therefore, by integrating the above expression and multiplying by 2,
   \begin{align} \label{eq:h-energy-1}
     &\mu |h'(t)|^2 + 2(1 - 2\mu \gamma) \int_0^t |h'(s)|^2 ds + (\mu \gamma^2 - \gamma +\alpha_k)|h(t)|^2 \nonumber \\
     & = \mu|\gamma u + v|^2 + (\mu \gamma^2 -\gamma + \alpha_k)|u|^2.
   \end{align}

   We derive a second energy estimate based on the fact that
   \begin{align*}
     &\frac{d}{dt} |\mu h'(t) + (1-2\mu \gamma)h(t)|^2 \\
     &= 2(\mu h''(t) + (1-2\mu \gamma)h'(t))(\mu h'(t) + (1-2\mu \gamma) h(t))\\
     &= -2(\mu\gamma^2-\gamma + \alpha_k) h(t)(\mu h'(t) + (1-2\mu \gamma) h(t)).
   \end{align*}
   The last equality is a consequence of \eqref{eq:h-eq}.
   Integrating both sides,
   \begin{align} \label{eq:h-energy-2}
     &|\mu h'(t) + (1 - 2\mu \gamma)h(t)|^2 + 2(\mu \gamma^2 - \gamma + \alpha_k)(1-2\mu \gamma)\int_0^t |h(s)|^2 ds
     \nonumber\\
     &\qquad+ \mu(\mu\gamma^2 -\gamma + \alpha_k)|h(t)|^2 \nonumber\\
     &= | \mu (\gamma u + v) + (1-2\mu \gamma) u|^2 +\mu(\mu\gamma^2 -\gamma + \alpha_k)|u|^2.
   \end{align}

   If $1-4\mu\alpha_k \geq 0$, we set $\gamma =\alpha_k$. This choice guarantees that the coefficients in \eqref{eq:h-eq} are positive. Specifically,
   \begin{equation} \label{eq:coeff-bounds-geq-0}
     \mu \gamma^2 -\gamma + \alpha_k = \mu \alpha_k^2 >0 \text{ and } 1-2\mu\gamma = \frac{1}{2} + \frac{1}{2}(1 - 4 \mu \alpha_k) \geq \frac{1}{2}.
   \end{equation}
   Then according to \eqref{eq:h-energy-1}, if $u=0$
   \[|h'(t)| \leq |v|\]
    and by the triangle inequality, \eqref{eq:h-energy-2}, and  the previous display,
   \[(1-2\mu\alpha_k)|h(t)| \leq \mu |h'(t)| + |\mu h'(t) + (1-2\mu\alpha_k)h(t)| \leq 2\mu|v|. \]
   Then by \eqref{eq:coeff-bounds-geq-0},
   \[|h(t)| \leq \frac{2\mu|v|}{1-2\mu\alpha_k}\leq  4\mu |v|.\]
    We chose $h(t) = e^{-\alpha_k t} f(t)$. It follows that $|f(t)| \leq 4 \mu |v| e^{-\alpha_k t}$ which is \eqref{eq:f-bound-alpha}. Similarly, $h'(t) = \alpha_k f(t)e^{\alpha_k t} + f'(t) e^{\alpha_k t}$. Therefore,
   \[|f'(t)| \leq \alpha_k |f(t)| + e^{-\alpha_k t} |h'(t)|  \leq (4\mu\alpha_k +1)|v|e^{-\alpha_k t}  \]
   In this regime $4\mu\alpha_k \leq 1$ so we can conclude that \eqref{eq:f'-bound-alpha} holds.

   Now we study the case where $1 - 4\mu\alpha_k < 0$. In this case we set $\gamma = \frac{1}{4\mu}$. Then
   \begin{equation} \label{eq:lower-bounds}
     1 - 2\mu\gamma = \frac{1}{2} \text{ and }
     \mu\gamma^2 -\gamma + \alpha_k = \alpha_k - \frac{3}{16\mu} \geq \frac{\alpha_k}{4}
   \end{equation}
   because $\frac{3}{16\mu}  \leq \frac{3\alpha_k}{4}$.
   If $u =0$, then by \eqref{eq:h-energy-1},
   \[|h(t)| \leq \frac{\sqrt{\mu}}{\sqrt{\mu\gamma^2 -\gamma + \alpha_k}}|v|.\]
   and
   \[|h'(t)| \leq |v|.\]
   Therefore by \eqref{eq:lower-bounds},
   \[|f(t)| \leq \sqrt{\frac{4\mu}{\alpha_k}}|v|e^{-\frac{t}{4\mu}}\]
   and
   \[|f'(t)| \leq \frac{1}{4\mu}|f(t)| + |h'(t)|e^{-\frac{t}{4\mu}} \leq \left(\frac{1}{\sqrt{4\mu\alpha_k}} + 1\right)|v|e^{-\frac{t}{2\mu}} \leq 2|v|e^{-\frac{t}{4\mu}}\]
   because $4\mu\alpha_k >1$.
   This proves \eqref{eq:f-bound-mu} and \eqref{eq:f'-bound-mu}.

   Finally, \eqref{eq:f-f'-bounds} is a consequence of \eqref{eq:h-energy-1} with $\gamma =0$.
 \end{proof}

 \begin{lemma} \label{lem:pi1-Smu-I-mu-bounded}
   For any $t\geq0$ and $\mu>0$ it holds that
   \begin{equation} \label{eq:Pi-Smu-I-mu-bounded}
      \|\Pi_1 S_\mu(t) \mathcal{I}_\mu \|_{\mathscr{L}(H)} \leq 4.
   \end{equation}
 \end{lemma}

 \begin{proof}
   This is an immediate consequence of \eqref{eq:f-bound-alpha} and \eqref{eq:f-bound-mu}. By \eqref{eq:semigroup-decomp}, $\Pi_1 S_\mu(t) \mathcal{I}_\mu e_k = f_k^\mu(t;0,1/\mu) e_k$. The $e_k$ are a complete orthonormal basis of $H$ and are eigenfunctions of $\Pi_1 S_\mu(T)\mathcal{I}_\mu$ and therefore
   \[\|\Pi_1 S_\mu(t) \mathcal{I}_\mu \|_{\mathscr{L}(H)} \leq \sup_{k \in \mathbb{N}} |f_k(t;0,1/\mu)|.\]
   For $k$ satisfying $1- 4\mu\alpha_k \geq 0$, \eqref{eq:f-bound-alpha} implies that
   $|f_k(t;0,1/\mu)| \leq 4$. For $k$ satisfying $1-4\mu\alpha_k < 0$,  \eqref{eq:f-bound-mu} implies that
   $|f_k^\mu(t;0,1/\mu)| \leq \frac{\sqrt{4}}{\sqrt{\mu\alpha_k}}$. For these $k$, $\mu\alpha_k > \frac{1}{4}$ and we can conclude that
   \[\|\Pi_1 S_\mu(t) \mathcal{I}_\mu \|_{\mathscr{L}(H)} \leq 4.\]
 \end{proof}

\begin{lemma} \label{lem:Pi1-Smu-I-0-op-norm}
 For any $\mu>0$ and $t\geq 0$,
 \begin{equation} \label{eq:Pi_1-sup}
   \left\|\Pi_1 S_\mu(t) \begin{pmatrix}I\\0 \end{pmatrix} \right\|_{\mathscr{L}(H)} \leq 1.
 \end{equation}
\end{lemma}
\begin{proof}
  This is an immediate consequence of \eqref{eq:f-bound-sup} because
  \[\left\|\Pi_1 S_\mu(t) \begin{pmatrix}I\\0 \end{pmatrix} \right\|_{\mathscr{L}(H)} = \sup_{k \in \mathbb{N}} |f_k^\mu(t;1,0)| \leq 1.\]
\end{proof}

\begin{lemma} \label{lem:Pi1Smu-op-norm}
  Let $N_\mu = \max\{ k \in \mathbb{N} : 1-4\mu\alpha_k \geq 0\}$. Then for any $t\geq 0$,
  \begin{equation} \label{eq:PNPi_2-sup}
    \left\|\Pi_1 S_\mu(t) \begin{pmatrix}0 \\ P_{N_\mu} \end{pmatrix}  \right\|_{\mathscr{L}(H)} \leq 4\mu.
  \end{equation}
\end{lemma}
\begin{proof}
  By \eqref{eq:f-bound-alpha},
  \[\left\|\Pi_1 S_\mu(t) \mathcal{I}_1 P_{N_\mu}  \right\|_{\mathscr{L}(H)} \leq \sup_{k \leq N_\mu} |f_k^\mu(t;0,1)| \leq 4\mu.\]
\end{proof}

\begin{lemma} \label{lem:Smu-I-PN-op-norm}
  Let $N_\mu = \max\{k \in \mathbb{N}: 1 - 4\mu\alpha_k \geq 0\}$. Then for any $t\geq0$,
  \begin{equation} \label{eq:I-PNP_2-sup}
    \left\|\Pi_1 \Smu(t) \begin{pmatrix}  0 \\ (I-P_{N_\mu})  \end{pmatrix} \right\|_{\mathscr{L}(H^{-1},H)} \leq \sqrt{4\mu}.
  \end{equation}
\end{lemma}
\begin{proof}
  Because of the presence of the $(I-P_{N_\mu})$ projection and fact that of $\Pi_1 S_\mu(t)\mathcal{I}_1 e_k = f_k^\mu(t;0,1)e_k$,
  \[\left\|\Pi_1 S_\mu(t) \mathcal{I}_1 (I-P_{N_\mu})\right\|_{\mathscr{L}(H^{-1},H)} = \sup_{k >N_\mu} \sqrt{\alpha_k}|f_k^\mu(t;0,1)|. \]
  By \eqref{eq:f-bound-mu},
  \[\left\|\Pi_1 S_\mu(t) \mathcal{I}_1 (I-P_{N_\mu})\right\|_{\mathscr{L}(H^{-1},H)} \leq\sqrt{4\mu}.\]
\end{proof}

\begin{lemma} \label{lem:Smu-pair-bound}
  For any $\mu\in(0,1)$ and $t\geq 0$, it holds that
  \begin{equation}
    \| \Smu(t)\|_{\mathscr{L}(\H)} \leq \mu^{-1/2}.
  \end{equation}
\end{lemma}
\begin{proof}
  Because $\mu \in (0,1)$ and the definition of $\H$, for any $(u,v) \in \H$ and $t \geq 0$,
  \[\mu |\Smu(t)(u,v)|_\H^2 \leq \mu |\Pi_2 \Smu(t)(u,v) |_{H^{-1}}^2 + |\Pi_1 \Smu(t)(u,v)|_H^2.\]
  By the Fourier decomposition \eqref{eq:semigroup-decomp}, right-hand side of the above expression equals
  \[\sum_{k=1}^\infty  \left(\frac{\mu}{\alpha_k}|(f_k^\mu)'(t;u_k,v_k)|^2 + |f_k^\mu(t;u_k,v_k)|^2 \right)\]
  where $u_k = \left<u,e_k\right>_H$ and $v_k = \left<v,e_k \right>_H$.
  It follows from \eqref{eq:f-f'-bounds} that the above expression is bounded by
  \[\leq \sum_{k=1}^\infty \left( \frac{\mu}{\alpha_k}|v_k|^2 + |u_k|^2 \right).\]
  Because $\mu \in (0,1)$, this is bounded by
  \[\leq \sum_{k=1}^\infty \left(\frac{1}{\alpha_k}|(f_k^\mu)'(t;u,v)|^2 + |f_k^\mu(t;u,v)|^2  \right) \leq |(u,v)|_\H^2.\]
  Therefore, for any $(u,v) \in \H$,
  \[|\Smu(t)(u,v)|_\H^2 \leq \frac{1}{\mu} |(u,v)|_\H^2,\]
  proving the result.
\end{proof}

Now we study the convegence of the Fourier coefficients $f_k^\mu(t;u,v)$ as $\mu \to 0$.
\begin{theorem}[Convergence] \label{thm:f-conv}
  Let $f_k^\mu(t;u,v)$ solve \eqref{eq:f_k-def}.
  \begin{enumerate}
    \item For any $k \in \mathbb{N}$, $T>0$, and $u \in \mathbb{R}$,
    \begin{equation} \label{eq:f-conv-u}
      \lim_{\mu \to 0} \sup_{t \in [0,T]}|f_k^\mu(t;u,0) - u e^{-\alpha_k t}| = 0.
    \end{equation}
    \item For any $ k \in \mathbb{N}$, $T>0$ $t_0 \in (0,T]$, and $v \in \mathbb{R}$,
    \begin{equation} \label{eq:f-conv-v}
      \lim_{\mu \to 0} \sup_{t \in [t_0,T]}| f_k^\mu(t;0,v/\mu) - ve^{-\alpha_k t}|=0.
    \end{equation}
%    \item For any $k \in \mathbb{N}$, $t>0$, and $v \in \mathbb{R}$,
%    \begin{equation} \label{eq:f'-conv-1/mu}
%      \lim_{\mu \to 0} (f_k^\mu)'(t;0,v/\mu) = -\alpha_k v e^{-\alpha_k t}.
%    \end{equation}
    \item For any $k \in \mathbb{N}$, $t>0$, and $v \in \mathbb{R}$,
    \begin{equation} \label{eq:f'-conv-0}
      \lim_{\mu \to 0} (f_k^\mu)'(t;0,v) = 0.
    \end{equation}
  \end{enumerate}
\end{theorem}

\begin{proof}
  One can prove each of these directly from the explicit formulas in \cite[Proposition 2.2]{cf-2006-add}. Below we present an alternative proof based on some arguments from \cite{f-2004}. Let $f^\mu_k(t) = f_k^\mu(t;u,v)$. Then because $\mu (f^\mu_k)''(t) + (f^\mu_k)'(t) + \alpha_k f^\mu_k(t) = 0$,
  \[\frac{d}{dt}\left(\mu e^{\frac{t}{\mu}}(f^\mu_k)'(t) \right) = -\alpha_k e^{\frac{t}{\mu}} f^\mu_k(t).\]
  Integrating both sides,
  \[\mu e^{\frac{t}{\mu}}(f^\mu_k)'(t) = \mu v -\alpha_k \int_0^t e^{\frac{s}{\mu}}f^\mu_k(s)ds \]
  and
  \begin{equation} \label{eq:f'-formula}
    (f^\mu_k)'(t) = v e^{-\frac{t}{\mu}} - \frac{\alpha_k}{\mu} \int_0^t e^{-\frac{(t-s)}{\mu}}f^\mu_k(s)ds.
  \end{equation}
  Integrating once more and changing the order of integration,
  \begin{equation} \label{eq:f-formula}
    f^\mu_k(t) = u + \mu v (1- e^{-\frac{t}{\mu}}) - \alpha_k \int_0^t (1 - e^{-\frac{(t-s)}{\mu}})f^\mu_k(s)ds.
  \end{equation}
  If $v =0$ and a limit $f_k^\mu(t) \to \bar{f}_k(t)$ exists, then the limit must solve
  \[\bar{f}_k(t) = u - \alpha_k \int_0^t \bar{f}_k(s)ds,\]
  the unique solution of which is $\bar{f}_k(t) = u e^{-\alpha_k t}$. To prove that $f_k^\mu(t)$ converges to $\bar{f}_k$, set $g^\mu_k(t) = f^\mu_k(t) - \bar{f}_k(t)$. Then
  \[g^\mu_k(t) = \alpha_k \int_0^t e^{-\frac{(t-s)}{\mu}}f^\mu_k(s)ds - \alpha_k \int_0^t g^\mu_k(s)ds. \]
  A standard Gr\"onwall along with the estimate \eqref{eq:f-bound-sup} proves that
  \[\sup_{t \in [0,T]} |g_k^\mu(t)| \leq \mu \alpha_k |u| e^{\alpha_k T}\] and consequently
  $\sup_{t \in [0,T]}|g^\mu_k(t)| \to 0$ and \eqref{eq:f-conv-u} follows.

  We can use a similar argument to show \eqref{eq:f-conv-v}. If $u = 0$ and $v = \frac{1}{\mu}$ in \eqref{eq:f-formula}, then
  \[f^\mu_k(t) = (1 - e^{-\frac{t}{\mu}}) - \alpha_k \int_0^t (1 -  e^{-\frac{(t-s)}{\mu}})f^\mu_k(s)ds.\]
  Let $\bar{f}(t) = e^{-\alpha_k t}$ and note that $\bar{f}(t) = 1 -\alpha_k \int_0^t \bar f(s)ds$. Setting $g^\mu_k(t) = f^\mu_k(t) - \bar f(t)$, we see that $g^\mu_k$ solves
  \[g^\mu_k(t) = -e^{-\frac{t}{\mu}} + \alpha_k \int_0^t e^{-\frac{(t-s)}{\mu}}f^\mu_k(s)ds - \alpha_k \int_0^t g^\mu_k(s)ds.\]
  If $\mu>0$ is small enough that $1-4\mu\alpha_k>0$, then \eqref{eq:f-bound-alpha} implies that for any $t>0$ $|f^\mu_k(t)| \leq 4$. Therefore, $\left|\int_0^t e^{-\frac{(t-s)}{\mu}} f^\mu_k(s)ds \right| \leq 4\mu.$
  By Gr\"onwall's inequality,
  \begin{align*}
    |g^\mu(t)| &\leq e^{-\frac{t}{\mu}} + 4\mu\alpha_k + \alpha_k\int_0^t \left(e^{-\frac{s}{\mu}} + 4\mu\alpha_k  \right)e^{\alpha_k (t-s)}ds\\
    &\leq e^{-\frac{t}{\mu}} + 5\mu\alpha_k e^{\alpha_k t} .
  \end{align*}
  Therefore, for any $0<t_0<T$,
  \[\sup_{t \in [t_0,T]} |g^\mu(t)| = 0\]
  and \eqref{eq:f-conv-v} follows for $v=1$. For general $v \in \mathbb{R}$, simply multiply both $f^\mu_k$ and $\bar f$ by $v$.

%  By \eqref{eq:f-bound-alpha} for  $\mu< \frac{1}{4\alpha_k}$, $|f^\mu_k(t)| \leq 4e^{-\alpha_k t}$. Therefore,
%  \[\int_0^t e^{-\frac{(t-s)}{\mu}}f^\mu_k(s)ds \to 0.\]
%  Consequently, $f^\mu_k(t) \to  e^{-\alpha_k t}$ and \eqref{eq:f-conv-v} follows by multiplying everything by $v \in \mathbb{R}$.

  Finally, we let $u=0$ and $v \in \mathbb{R}$ in \eqref{eq:f'-formula}. Then
  \[(f^\mu_k)'(t) = v e^{-\frac{t}{\mu}} - \frac{\alpha_k}{\mu} \int_0^t e^{-\frac{(t-s)}{\mu}}f^\mu_k(s)ds. \]
  By \eqref{eq:f-bound-alpha}, for $\mu < \frac{1}{4\alpha_k}$, $|f^\mu_k(s)|\leq 4 \mu |v| e^{-\alpha_k s}$. Therefore,
  \[|(f^\mu_k)'(t)| \leq |v| e^{-\frac{t}{\mu}} + 4\alpha_k|v| \int_0^t e^{-\frac{(t-s)}{\mu}}ds \to 0.\]
\end{proof}

\section{Regularity of the stochastic convolution} \label{S:stoch-conv}
Let $G$ be the operator defined in \eqref{eq:G-def} and let $\varphi(t)$ and $\psi(t)$ be some $H$-valued processes that are adapted to the natural filtration of $w(t)$. In this section we study the stochastic convolution processes
\[\int_0^t \Smu(t-s) \I_\mu  G(s,\varphi(s))Qdw(s)\]
and the differences
\[\int_0^t \Smu(t-s) \I_\mu [G(s,\varphi(s)) - G(s,\psi(s))]Qdw(s).\]
In order to study both of these objects at the same time and to simplify our notation, for the rest of this chapter we will let $\Phi(t)$ denote either $G(\varphi(t))$ or $G(\varphi(t)) - G(\psi(t))$.

Before establishing estimates on the stochastic convolution we discuss the properties of such a $\Phi$. For any $t\geq 0$, $\Phi(t)$ is a bounded linear operator from $L^\infty(D)$ to $H$. $\Phi(t)$ is also a bounded linear operator from $H$ to $L^1(D)$.

If  $\varphi(t) \in H$, and $h \in L^\infty(D)$ then by the linear growth of $g$ in Assumption \ref{assum:bg},
\begin{align*}
  &|G(t,\varphi(t))h|_H^2 = \int_D |g(t,x,\varphi(t,x))h(x)|^2 dx \leq C \int_D \left(1 + |\varphi(t,x)|^2 \right)^2|h(x)|^2dx \\
  &\leq C (1 + |\varphi(t)|_H^2)|h|_{L^\infty(D)}^2.
\end{align*}
If $\varphi(t) \in H$ and $h \in H$, then
\begin{align*}
  &|G(t,\varphi(t))h|_{L^1(D)} = \int_D |g(t,x,\varphi(t,x))h(x)|dx \\
  &\leq \left(\int_D |g(t,x,\varphi(t,x))|^2 dx \right)^{\frac{1}{2}} \left( \int_D |h(x)|^2 dx\right)^{\frac{1}{2}} \leq C \left(1 + |\varphi(t)|_H\right)|h|_H.
\end{align*}
Similarly, if $\Phi(t) = (G(t,\varphi(t)) - G(t, \psi(t)))$, and $\varphi(t), \psi(t) \in H$ and $h \in L^\infty(D)$,
\begin{align} \label{eq:G-Lip}
  &|(G(t,\varphi(t)) - G(t,\psi(t)))h|_H^2= \int_D |(g(t,x,\varphi(t,x)) - g(t,x,\psi(t,x)))h(x)|^2 dx\nonumber\\
   &\leq C\int_D|\varphi(t,x) - \psi(t,x)|^2|h(x)|^2 dx
   \leq C|\varphi(t) - \psi(t)|_H^2 |h|_{L^\infty(D)}^2
\end{align}
and if $h \in H$, then
\begin{align*}
  &|(G(t,\varphi(t)) - G(t,\psi(t)))h|_{L^1(D)} \leq C|\varphi(t) - \psi(t)|_H |h|_H.
\end{align*}

Let $\Phi^\star(t)$ denote the adjoint of $\Phi(t)$ in $H$ in the sense that if $h_1 \in L^\infty(D)$ and $h_2 \in H = L^2(D)$ or $h_1 \in H$ and $h_2 \in L^\infty(D)$,
\[\left< \Phi(t) h_1, h_2 \right>_H = \left<h_1, \Phi^\star(t) h_2 \right>_H.\]
Because we chose $\Phi(t) = G(t,\varphi(t))$ or $\Phi(t) = (G(t,\varphi(t))-G(t,\psi(t)))$, $\Phi(t) = \Phi^\star(t)$. Notice that if $\Phi(t) = G(t,\varphi(t))$, $h_1\in L^\infty(D)$ and $h_2 \in H$,
\[\left< \Phi(t)h_1, h_2\right>_H = \int_D g(t,x,\varphi(t,x))h_1(x)h_2(x)dx = \left<h_1, \Phi^\star(t)h_2 \right>_h.  \]

In this way, $\Phi(t)$ is a self-adjoint $\mathscr{L}(L^\infty(D),H) \cap \mathscr{L}(H, L^1(D))$-valued process that is adapted to the natural filtration of $w(t)$. We define the stochastic convolution
\begin{equation} \label{eq:Gamma-def}
  \Gamma^\mu(t) = \int_0^t \Smu(t-s)\I_\mu \Phi(s)Qdw(s).
\end{equation}

By the stochastic factorization formula \cite[Chapter 5.3.1]{dpz}, for $0< \alpha< 1$ to be chosen later,
\begin{equation} \label{eq:stoch-conv}
  \Gamma^\mu(t) = \frac{\sin(\alpha \pi)}{\pi} \int_0^t (t-s)^{\alpha - 1} \Smu(t-s)\Gamma_\alpha^\mu(s)ds
\end{equation}
where
\begin{equation} \label{eq:Gamma-mu-alpha}
  \Gamma_\alpha^\mu(t) = \int_0^t (t-s)^{-\alpha} \Smu(t-s)\I_\mu \Phi(s)dw(s).
\end{equation}

We begin with estimates on $\Gamma_\alpha^\mu$.

\begin{lemma} \label{lem:Pi_1-stoch-conv}
  Let $1 <q < \frac{d}{d-2}$ satisfy \eqref{eq:lambda-sum}. Let $0<2\alpha < 1- \frac{d(q-1)}{2q}$. Then for any $p\geq \frac{1}{\alpha}$ and $T>0$,  there exists a constant $C = C(\alpha,p,T,Q)$ independent of $\mu$ such that for any $t \in [0,T]$,
  \begin{equation} \label{eq:Pi_1-Gamma_alpha}
    \E \left|\Pi_1\Gamma_\alpha^\mu(t) \right|_H^p \leq C \E \sup_{s\in [0, t]}\|\Phi(s)\|_{\mathscr{L}(L^\infty(D),H)}^p.
  \end{equation}`
\end{lemma}
\begin{proof}
  By the Burkholder-Davis-Gundy inequality \cite[Theorem 4.36]{dpz},
  \begin{equation} \label{eq:bdg}
    \E \left|\Pi_1 \Gamma_\alpha^\mu(t) \right|_H^p \leq C \E \left(\sum_{j=1}^\infty \int_0^t (t-s)^{-2\alpha} |\Pi_1 \Smu(t-s) \Imu \Phi(s)Q e_j|_H^2 ds \right)^{p/2}
  \end{equation}
  where $\{e_j\}$ is the complete orthonormal basis of $H$ that diagonalizes $Q$ and $A$ in Assumptions \ref{assum:A} and \ref{assum:Q}.

  For the rest of the proof, it is enough to study the quadratic variation.
  \[\Lambda^\mu_\alpha(t) := \sum_{j=1}^\infty \int_0^t (t-s)^{-2\alpha} |\Pi_1 \Smu(t-s) \Imu \Phi(s)Q e_j|_H^2 ds.\]
  We expand this expression into a double sum
  \begin{align*}
    &\sum_{k=1}^\infty \sum_{j=1}^\infty \int_0^t (t-s)^{-2\alpha} \left< \Pi_1 \Smu(t-s) \Imu \Phi(s)Q e_j, e_k  \right>_H^2 ds\\
    &= \sum_{k=1}^\infty \sum_{j=1}^\infty \int_0^t (t-s)^{-2\alpha} \left<  \Phi(s)Q e_j,\I_\mu^\star \Smustar(t-s) \Pi_1^\star e_k  \right>_H^2 ds.
  \end{align*}
  Notice that for any $k, j \in \mathbb{N}$ and $t\geq 0$
  \begin{align*}
    &\left< \I_\mu^\star \Smustar(t) \Pi_1^\star e_k, e_j \right>_H = \left<\Pi_1 \Smu(t) \Imu e_j, e_k \right>_H \\
    &= \begin{cases} f_k^\mu(t) & \text{ if } j=k\\0 & \text{ otherwise} \end{cases}
  \end{align*}
  where $f_k^\mu(t) = f_k^\mu(t;0,1/\mu)$ solves \eqref{eq:f_k-def} with $u_k=0$ and $v_k = 1/\mu$.
  Therefore, along with the fact that $Q e_j = \lambda_j e_j$, the quadratic variation can be written as
  \[\Lambda^\mu_\alpha(t) = \sum_{k=1}^\infty \sum_{j=1}^\infty \int_0^t (t-s)^{-2\alpha} \lambda_j^2 (f_k^\mu(t-s))^2 \left<\Phi(s)e_j, e_k \right>_H^2 ds.\]
  Apply H\"older's inequality with exponents $\frac{q}{2}$ and $\frac{q}{q-2}$ to the double sum,
  \begin{align*}
    \Lambda^\mu_\alpha(t) \leq \int_0^t &(t-s)^{-2\alpha} \left(\sum_{k=1}^\infty \sum_{j=1}^\infty \lambda_j^{q} \left<\Phi(s)e_j,e_k \right>_H^2 \right)^{2/q}\\
    &\times \left(\sum_{k=1}^\infty \sum_{j=1}^\infty (f_k^\mu(t-s))^{2q/(q-2)} \left<e_j, \Phi^\star(s) e_k \right>_H^2 \right)^{(q-2)/q}ds\\
    = \int_0^t &(t-s)^{-2\alpha}  \left(\sum_{j=1}^\infty \lambda_j^{q} |\Phi(s) e_j|_H^2 \right)^{2/q}\\
    &\times
     \left(\sum_{k=1}^\infty (f_k^\mu(t-s))^{2q/(q-2)} |\Phi^\star(s)e_k|_H^2 \right)^{(q-2)/q}ds.
  \end{align*}
  We will denote $\|Q\|_{q} :=  \left(\sum_{j=1}^\infty \lambda_j^{q}  \right)^{1/q}$ as in \cite{c-2003}. In Assumption \ref{assum:A} we assumed that the $e_k$ are uniformly bounded in $L^\infty(D)$. It follows that
  $$\sup_j |\Phi(s)e_j|_H = \sup_k |\Phi^\star(s)e_k|_H \leq C\|\Phi(s)\|_{\mathscr{L}(L^\infty(D),H)}.$$ Therefore,
  \begin{equation} \label{eq:quad-var-upper-1}
    \Lambda^\mu_\alpha(t) \leq \int_0^t (t-s)^{-2\alpha} \|Q\|_{q}^2 \|\Phi(s)\|_{\mathscr{L}(L^\infty(D),H)}^2 \left( \sum_{k=1}^\infty (f_k^\mu(t-s))^{2q/(q-2)}\right)^{(q-2)/q} ds.
  \end{equation}

  We analyze the sum
  \[\left(\sum_{k=1}^\infty (f_k^\mu(t))^{2q/(q-2)} \right)^{(q-2)/q}\]
  by dividing it into two pieces. Let $N_\mu = \max\{k:  1-4\mu\alpha_k \geq 0\}.$ Then by \eqref{eq:f-bound-alpha} and \eqref{eq:f-bound-mu} with $v=1/\mu$
  \begin{align*}
    &\left(\sum_{k=1}^\infty (f_k^\mu(t))^{2q/(q-2)} \right)^{(q-2)/q} \\
    &\leq C\left(\sum_{k=1}^{N_\mu} e^{-2\alpha_k qt/(q-2) }
      + \sum_{k=N_\mu + 1}^\infty (\mu \alpha_k)^{-q/(q-2)}e^{-\frac{tq}{2(q-2)\mu}} \right)^{(q-2)/q}.
  \end{align*}
  For any $x,y\geq0$ it follows that $(x+y)^{(q-2)/q} \leq x^{(q-2)/q} + y^{(q-2)/q}$. Therefore, the above expression is bounded by
  \begin{align*}
    &\leq C\left( \sum_{k=1}^{N_\mu} e^{-2q\alpha_k t /(q-2)} \right)^{(q-2)/q} + \frac{Ce^{-\frac{t}{2\mu}}}{\mu}\left(\sum_{k=N_\mu +1}^\infty \alpha_k^{-q/(q-2)} \right)^{(q-2)/q}\\
    &:=J_1 + J_2.
  \end{align*}

  The finite sum $J_1$ behaves like the eigenfunctions of the semigroup in the parabolic case considered in \cite{c-2003}. Because $\alpha_k \sim k^{2/d}$, we can choose $r>d/2$ to be specified later such that $\sum_{k=1}^\infty \alpha_k^{-r}<+\infty.$ Additionally, there exists a constant such that for all $k \in \mathbb{N}$, $e^{-\alpha_k t} \leq C\frac{1}{\alpha_k^r t^r}$. It follows that
  \[J_1 \leq C\left( \sum_{k=1}^{N_\mu} e^{-2q\alpha_k t /(q-2)} \right)^{(q-2)/q} \leq C\left(\sum_{k=1}^\infty \frac{1}{\alpha_k^r (t-s)^r}  \right)^{(q-2)/q}.\]
  Therefore, $J_1 \leq C (t-s)^{-r(q-2)/q}$ by \eqref{eq:alpha-sum}.

  We show that the tail sum $J_2$ is small. By Assumption \ref{assum:A}, there exists $C$ such that $\frac{1}{C} k^{2/d} \leq \alpha_k \leq Ck^{2/d}$.
   By the definition of $N_\mu$, $\alpha_{(N_\mu + 1)} > \frac{1}{4\mu}$. This means that $\frac{1}{4\mu} < \alpha_{(N_\mu+1)} \leq C(N_\mu+1)^{2/d}$ so $N_\mu +1 > \frac{1}{C} \mu^{-d/2}$.
  Therefore,
  \begin{align} \label{eq:sum-decay-rate}
    &\sum_{k=N_\mu +1}^\infty {\alpha_k^{-q/(q-2)}} \leq \sum_{k= [C^{-1}\mu^{-d/2}]}^\infty \frac{C}{k^{2q/(d(q-2))}}\nonumber\\
    &\leq C\int_{C^{-1}\mu^{-d/2}}^\infty x^{-2q/(d(q-2))}dx \leq  C \mu^{\frac{q}{q-2} - \frac{d}{2}}.
  \end{align}
  This means that
  \[J_2 \leq C\frac{e^{-\frac{t}{2\mu}}}{\mu} \mu^{1 - \frac{d(q-2)}{2q}} \leq C\mu^{ - \frac{d(q-2)}{2q}}e^{-\frac{t}{2\mu}}. \]

  Plugging this back into \eqref{eq:quad-var-upper-1},
  \begin{align*}
    &\Lambda^\mu_\alpha(t) \leq C\|Q\|_{2q}^2 \int_0^t(t-s)^{-2\alpha} \left((t-s)^{-r(q-2)/q} + \mu^{-\frac{d(q-2)}{2q}}e^{-\frac{t-s}{2\mu}} \right) \|\Phi(s)\|_{\mathscr{L}(L^\infty(D),H)}^2ds\\
    &\leq C \|Q\|_{2q}^2 \sup_{s\in [0, t]} \|\Phi(s)\|_{\mathscr{L}(L^\infty(D),H)} \int_0^t \left((t-s)^{-2\alpha - r(q-2)/q} + \mu^{- \frac{d(q-2)}{2q}}(t-s)^{-2\alpha} e^{-\frac{t-s}{2\mu}} \right)ds.
  \end{align*}
  By a change of variables,
  \begin{equation} \label{eq:change-of-vars}
    \int_0^\infty s^{-2\alpha} e^{-\frac{s}{2\mu}}ds = (2\mu)^{1-2\alpha}\int_0^\infty t^{-2\alpha} e^{-t}dt  = C \mu^{1-2\alpha}.
  \end{equation}
  From these estimates we see that
  \[\Lambda^\mu_\alpha(t) \leq C \|Q\|_{2q}^2 \sup_{s\in [0,t]} \|\Phi(s)\|_{\mathscr{L}(L^\infty(D),H)} \left(\int_0^t (t-s)^{-2\alpha - r(q-2)/q}ds + \mu^{1 - 2\alpha - \frac{d(q-2)}{2q}} \right).\]
  We chose $r>d/2$ and $\alpha$ to satisfy, $-2\alpha - r(q-2)/q >-1$. It follows  that $1 - 2\alpha - \frac{d(q-2)}{2q}\geq 0$. Therefore, there exists a constant $C>0$ independent of $\mu \in (0,1)$ such that
  \[\Lambda^\mu_\alpha(t) \leq C\|Q\|_{2q}^2 \sup_{s \leq t}\|\Phi(s)\|_{\mathscr{L}(L^\infty(D),H)}^2.  \]
  The result follows by the BDG inequality \eqref{eq:bdg}.
\end{proof}

Now we analyze the second component of $\Gamma_\alpha^\mu(t)$. This will diverge as $\mu \to 0$. It will be convenient to analyze the moments of $\Gamma_\alpha^\mu$ in two pieces. Let $N_\mu = \max\{k: 1-4\mu\alpha_k \geq 0\}$ as above. Let $P_{N_\mu}$ be the projection onto the span of the modes $\{e_1,.. e_{N_\mu}\}$.

\begin{lemma} \label{lem:Pi_2-stoch conv}
  Let $2 <q < \frac{2d}{d-2}$ satisfy \eqref{eq:lambda-sum}. Let $0<2\alpha < 1- \frac{d(q-2)}{2q}$. Let $\Gamma^\mu_\alpha$ be given by \eqref{eq:Gamma-alpha}. Then for any $p\geq 2$ and $T>0$,  there exist constants $C = C(p,T) >0$ and $\zeta = \zeta(p,T)>0$ such that
  \begin{enumerate}
    \item For any $t \in [0,T]$, and $\mu \in (0,1)$,
        \begin{equation} \label{eq:PNPi_2-Gamma_alpha}
          \E\left|P_{N_\mu}\Pi_2\Gamma_\alpha^\mu(t) \right|_H^p \leq \frac{C}{\mu^{p}} \E \sup_{s \in [0,t]} \|\Phi(s)\|_{\mathscr{L}(L^\infty(D),H)}^{p}.
    \end{equation}
  \item For any fixed $t \in [0,T]$,
        \begin{equation} \label{eq:Pi_2-stoch-conv-to-0}
          \lim_{\mu \to 0} \mu^p\E| P_{N_\mu} \Pi_2 \Gamma_\alpha^\mu(t)|_H^p = 0.
        \end{equation}
  \item  For any fixed $t \in [0,T]$ and $\mu \in (0,1)$,
       \begin{equation}\label{eq:I-PNP_2-Gamma_alpha}
         \E \left|(I - P_{N_\mu})\Pi_2 \Gamma_\alpha^\mu(t) \right|_{H^{-1}}^p \leq \frac{C}{\mu^{(p-\zeta)/2}} \E\sup_{s \in [0,t]} \|\Phi(s)\|_{\mathscr{L}(L^\infty(D),H)}^p.
       \end{equation}
  \end{enumerate}
\end{lemma}
\begin{proof}
  The proofs of this lemma are similar to the proof of Lemma \ref{lem:Pi_1-stoch-conv}. Let $\Lambda_1(t)$ be the quadratic variation of $P_{N_\mu}\Pi_2 \Gamma^\mu_\alpha$.
  \begin{align*}
    \Lambda_1(t) = &\sum_{j=1}^\infty \int_0^t (t-s)^{-2\alpha} |P_{N_\mu}\Pi_2\Smu(t-s)\Imu\Phi(s)Qe_j|_H^2ds\\
    = &\sum_{k=1}^{N_\mu} \sum_{j=1}^\infty \int_0^t (t-s)^{-2\alpha}\left<\Phi(s)Q e_j, \I_\mu^\star \Smustar(t-s) \Pi_2^\star e_k \right>_H^2ds.
  \end{align*}
  The eigenvalues satisfy $Qe_j = \lambda_j e_j$ and $\I_\mu^\star \Smustar(t-s)\Pi_2^\star e_k = (f^\mu_k)'(t-s)e_k$ where $f^\mu_k$ solves \eqref{eq:f_k-def} with $u_k=0$ and $v_k = 1/\mu$. Then
  \begin{equation} \label{eq:Lambda1-bound}
    \Lambda_1(t) \leq \sum_{k=1}^{N_\mu} \sum_{j=1}^\infty \int_0^t (t-s)^{-2\alpha} \lambda_j^2 |(f_k^\mu)'(t)|^2 \left<\Phi(s)e_j, e_k \right>_H^2ds.
  \end{equation}
  By \eqref{eq:f'-bound-alpha} with $v=\frac{1}{\mu}$, for $k \in \{1, ..., N_\mu\}$
  \[|(f^\mu_k)'(t)| \leq \frac{2e^{-\alpha_k t}}{\mu}.\]
  Therefore,
  \[\Lambda_1(t) \leq \frac{C}{\mu^2}\sum_{k=1}^{N_\mu} \sum_{j=1}^\infty \int_0^t (t-s)^{-2\alpha} \lambda_j^2 e^{-2\alpha_k (t-s)} \left< \Phi(s)e_j,e_k \right>_H^2 ds. \]
  By the H\"older inequality on the double sum and following the arguments of the proof of Lemma \ref{lem:Pi_1-stoch-conv},
  \[\Lambda_1(t) = \frac{C}{\mu^2} \int_0^t (t-s)^{-2\alpha} \|Q\|_{q}^2 \|\Phi(s)\|_{\mathscr{L}(L^\infty(D),H)}^2 \left(\sum_{k=1}^{N_\mu} e^{-2\alpha_kq(t-s)/(q-2)} \right)^{(q-2)/q}ds.\]
  Let $r>d/2$ satisfy $-2\alpha -(q-2)r/q > -1$, so that
  \[\sum_{k=1}^\infty e^{-2\alpha_k q (t-s)/(q-2)} \leq C (t-s)^{-r} \sum_{k=1}^{N_\mu} \alpha_k^{-r}.\]
  The sum $\sum_{k=1}^{N_\mu} \alpha_k^{-r} \leq \sum_{k=1}^\infty \alpha_k^{-r}<+\infty$.
  Therefore,
  \begin{align*}
    &\Lambda_1(t) \leq \frac{C}{\mu^2}\|Q\|_{q}^2 \sup_{s \in [0,t]} \|\Phi(s)\|_{\mathscr{L}(L^\infty(D),H)}^2 \int_0^t (t-s)^{-2\alpha - (q-2)r/q}ds\\
    & \leq \frac{C}{\mu^2}\|Q\|_{q}^2 \sup_{s \leq t} \|\Phi(s)\|_{\mathscr{L}(L^\infty(D),H)}^2.
  \end{align*}
  By the BDG inequality,
  \[\E \left|P_{N_\mu} \Pi_1 \Gamma_\alpha^\mu(t) \right|_H^p \leq \E (\Lambda_1(t))^{p/2}\]
  and \eqref{eq:PNPi_2-Gamma_alpha} follows.

  All of the previous calculations allow us to use a dominated convergence theorem to prove \eqref{eq:Pi_2-stoch-conv-to-0}. The upper bound for \eqref{eq:Lambda1-bound} using \eqref{eq:f'-bound-alpha} was established above.
  Specifically, for $k \in \{1,...,N_\mu\}$, $j \in \mathbb{N}$, and $s,t \in [0,T]$,
  \[ \mu^2|(f_k^\mu)'(t;0,1/\mu)|^2 \left<\Phi(s)e_j, e_k \right>_H^2 \leq C \lambda_j^2 e^{-2\alpha_k t} \left<\Phi(s)e_j, e_k \right>_H^2.\]
  Notice that $\mu (f_k^\mu)'(t,0,\frac{1}{\mu}) = (f_k^\mu)'(t;0,1)$. By \eqref{eq:f'-conv-0}, for each $s>0$, $k \leq N_\mu$, and $j \in \mathbb{N}$,
  \[\lim_{\mu \to 0} (t-s)^{-2\alpha} \lambda_j^2 \mu^2|(f_k^\mu)'(t;0,1/\mu)|^2 \left<\Phi(s)e_j, e_k \right>_H^2 =0.\]
  Therefore, by \eqref{eq:Lambda1-bound} and the dominated convergence theorem $\Lambda_1(t) \to 0$ with probability 1. Then by using the BDG inequality and one more application of the dominated convergence theorem, \eqref{eq:Pi_2-stoch-conv-to-0} follows.

  As for the higher modes, let
  \begin{align*}
    &\Lambda_2(t) = \sum_{j=1}^\infty \int_0^t (t-s)^{-2\alpha} |(I - P_{N_\mu}) \Pi_2 \Smu(t-s)\Imu \Phi(s)Q e_j|_{H^{-1}}^2ds\\
    &= \sum_{j=1}^\infty \int_0^t (t-s)^{-2\alpha} \left|(-A)^{-1/2}(I-P_{N_\mu})\Pi_2\Smu(t-s)\Imu \Phi(s)Q e_j\right|_H^2 ds.
  \end{align*}
  Expanding this to a double sum,
  \begin{align*}
    &\leq \sum_{k=N_\mu + 1}^\infty \sum_{j=1}^\infty  \int_0^t (t-s)^{-2\alpha} \left< \Phi(s)Q e_j, \I_\mu^\star \Smustar(t-s) \Pi_2^\star (I-P_{N_\mu})^\star (-A)^{-1/2} e_k\right>_H^2 ds
  \end{align*}
  Recognize that for $k,j \in \mathbb{N}$
  \begin{align*}
    &\left<\I_\mu^\star \Smustar(t-s) \Pi_2^\star (I-P_{N_\mu})^\star (-A)^{-1/2} e_k, e_j \right>_H\\
    &= \left<(-A)^{-1/2}(I-P_{N_\mu})\Pi_2\Smu(t-s)\Imu e_j, e_k \right>_H\\
    &= \begin{cases}\alpha_k^{-1/2} (f_k^\mu)'(t-s) & \text{ if } k=j >N_\mu,\\ 0 & \text{ otherwise}.  \end{cases}
  \end{align*}
  By \eqref{eq:f'-bound-mu},
  \[\alpha_k^{-1/2} |(f_k^\mu)'(t-s)| \leq C\alpha_k^{-1/2} \mu^{-1} e^{-\frac{t}{4\mu}}.\]
  By \eqref{eq:sum-decay-rate},
  \begin{align*}
    &\sum_{k=N_\mu +1}^\infty \left(\alpha_k^{-1}|(f_k^\mu)'(t-s)|^2 \right)^{q/(q-2)}\\
     &\leq  \sum_{k=N_\mu +1}^\infty \frac{C e^{-qt/(2\mu(q-2))}}{\mu^{2q/(q-2)} \alpha_k^{q/(q-2)}} \leq  Ce^{-qt/(2\mu(q-2))} \mu^{-q/(q-2) - \frac{d}{2}} .
  \end{align*}
  Applying the H\"older inequality to $\Lambda_2(t)$,
  \begin{align*}
    &\Lambda_2(t) \leq \frac{C}{\mu^{1+ \frac{d(q-2)}{2q}}} \int_0^t (t-s)^{-2\alpha} e^{-\frac{t-s}{2\mu}} \|Q\|_{q}^2 \|\Phi(s)\|^2_{\mathscr{L}(L^\infty(D),H)} ds
    \end{align*}
  By \eqref{eq:change-of-vars},
  \[\Lambda_2(t) \leq \frac{C}{\mu^{ 2\alpha + \frac{d(q-2)}{2q}}}\|Q\|_{q}^2 \sup_{s\in [0,t]}\|\Phi(s)\|_{\mathscr{L}(L^\infty(D),H)}^2. \]
  We chose $\alpha$ so that $2\alpha + \frac{d(q-1)}{2q} <1$. This means that there exists $\zeta>0$ such that
  \[\Lambda_2(t) \leq \frac{C}{\mu^{1-(\zeta/p)}} \sup_{s\in [0,t]} \|\Phi(s)\|_{\mathscr{L}(L^\infty(D),H)}^2.\]
  By the BDG inequality,
  \[\E \left|(I-P_{N_\mu})\Pi_2 \Gamma_\alpha^\mu(t) \right|_{H^{-1}}^p \leq \frac{C}{\mu^{(p-\zeta)/2}} \E \sup_{s \in [0,t]} \|\Phi(s)\|_{\mathscr{L}(L^\infty(D),H)}^p.\]
\end{proof}

Now we can establish a priori bounds on the supremum norm of the stochastic convolution.
\begin{theorem} \label{thm:stoch-conv-sup-norm}
  Let $\Gamma^\mu(t)$ be given by \eqref{eq:Gamma-def}. For any $p \geq 1$ and $T\geq 0$, there exists a constant $C= C(p,T)$ such that for all $\mu \in (0,1)$
  \begin{equation} \label{eq:stoch-conv-sup-norm}
    \E\sup_{t \in [0,T]} | \Pi_1 \Gamma^\mu(t)|_H^p \leq C\E \int_0^T \sup_{s \in [0,t]}\|\Phi(s)\|_{\mathscr{L}(L^\infty(D),H)}^pdt.
  \end{equation}
  Notice that this constant is independent of $\mu\in (0,1)$.
\end{theorem}
\begin{proof}
  We use the stochastic convolution formula \eqref{eq:stoch-conv},
  \[\Gamma^\mu(t) = \frac{\sin(\alpha \pi)}{\pi} \int_0^t (t-s)^{\alpha - 1} \Smu(t-s)\Gamma_\alpha^\mu(s)ds.\]
  We divide $\Gamma_\alpha^\mu$ into three different pieces.
  Recall that $\Pi_1, \Pi_2$ defined in \eqref{eq:Pi-defs} and $P_{N_\mu}, (I-P_{N_\mu})$ defined above Lemma \ref{lem:Pi_2-stoch conv} are all projections. We can rewrite the stochastic convolution formula \eqref{eq:stoch-conv} as
  \begin{align*}
    &\Gamma^\mu(t) = \frac{\sin(\alpha \pi)}{\pi}\int_0^t (t-s)^{\alpha-1} \Smu(t-s) \Bigg(
      \begin{pmatrix} I  \\ 0  \end{pmatrix}\Pi_1 \Gamma_\alpha^\mu(s) \\
      &+\begin{pmatrix}  0 \\ P_{N_\mu}  \end{pmatrix}P_{N_\mu}\Pi_2 \Gamma_\alpha^\mu(s) +\begin{pmatrix} 0  \\ (1-P_{N_\mu})  \end{pmatrix} (1-P_{N_\mu})\Pi_2 \Gamma^\mu_\alpha(s)
    \Bigg).
  \end{align*}

  %We use Lemma \ref{lem:f-bounds} to establish bounds on the supremum norms of the linear operators.
%  By \eqref{eq:f-bound-sup},
%  \begin{equation} \label{eq:Pi_1-sup}
%    \left\|\Pi_1 \Smu(t)\begin{pmatrix} I  \\ 0  \end{pmatrix} \right\|_{\mathscr{L}(H)} = \sup_{k\in \mathbb{N}} |f_k^\mu(t;1,0)| \leq 1.
%  \end{equation}
%  By \eqref{eq:f-bound-alpha},
%  \begin{equation} \label{eq:PNPi_2-sup}
%    \left\|\Pi_1 \Smu(t) \begin{pmatrix}  0 \\ P_{N_\mu}  \end{pmatrix} \right\|_{\mathscr{L}(H)} \leq \sup_{k\leq N_\mu} |f_k^\mu(t;0,1)| \leq 4\mu.
%  \end{equation}
  %By \eqref{eq:f-bound-mu},
%  \begin{equation} \label{eq:I-PNP_2-sup}
%    \left\|\Pi_1 \Smu(t) \begin{pmatrix}  0 \\ (I-P_{N_\mu})  \end{pmatrix} \right\|_{\mathscr{L}(H^{-1},H)} \leq \sup_{k>N_\mu} \sqrt{\alpha_k}{|f_k^\mu(t;0,1)|} \leq \sqrt{4\mu}.
%  \end{equation}
%  Notice that the third formula is considered as a linear operator from $H^{-1} \to H$. We can gain this extra spatial regularity due to the $\frac{1}{\sqrt{\alpha_k}}$ in \eqref{eq:f-bound-mu}.

  Choose $\alpha>0$ satisfying the assumptions of Lemmas \ref{lem:Pi_1-stoch-conv} and \ref{lem:Pi_2-stoch conv}. Let $p> \frac{1}{\alpha}$.

  Applying the H\"older inequality and using  \eqref{eq:Pi_1-sup} and \eqref{eq:Pi_1-Gamma_alpha},
  \begin{align*}
    &\E\sup_{t\in [0,T]}\left|\int_0^t (t-s)^{\alpha -1}\Pi_1 \Smu(t-s) \begin{pmatrix} I  \\ 0  \end{pmatrix} \Pi_1 \Gamma_\alpha^\mu(s)ds\right|_H^p\\
    &\leq  C\left(\int_0^T s^{(\alpha -1)p/(p-1)} \left\|\Pi_1 \Smu(s)\begin{pmatrix} I  \\ 0  \end{pmatrix} \right\|_{\mathscr{L}(H)}^{p/(p-1)}ds \right)^{p-1}\E\int_0^T |\Pi_1\Gamma_\alpha^\mu(s)|^p_H ds\\
    &\leq C\int_0^T \E \sup_{s \in [0,t]} \|\Phi(s)\|_{\mathscr{L}(L^\infty(D),H)}^p dt.
  \end{align*}
  The previous line follows because $p> \frac{1}{\alpha}$ implies $(\alpha-1)p/(p-1)>-1$.

  By the same argument with \eqref{eq:PNPi_2-sup} and \eqref{eq:PNPi_2-Gamma_alpha},
  \begin{align*}
    &\E \sup_{t\in [0, T]} \left| \int_0^t (t-s)^{\alpha -1} \Pi_1 \Smu(t-s)\begin{pmatrix}  0 \\ P_{N_\mu}  \end{pmatrix} P_{N_\mu} \Pi_2 \Gamma_\alpha^\mu(s)ds\right|_H^p\\
    & \leq C \left(\int_0^T s^{(\alpha-1)p/(p-1)} \left\| \Pi_1\Smu(s)\begin{pmatrix}  0 \\ P_{N_\mu}  \end{pmatrix} \right\|_{\mathscr{L}(H)}^{p/(p-1)}ds \right)^{p-1} \int_0^T | P_{N_\mu} \Pi_2 \Gamma_\alpha^\mu(s)|_H^pds\\
    &\leq C\mu^p \E \int_0^T |P_{N_\mu} \Pi_2\Gamma_\alpha^\mu(t)|_H^p dt \leq C \E \sup_{t \in [0,T]} \|\Phi(t)\|_{\mathscr{L}(L^\infty(D),H)}^pdt.
  \end{align*}

  By  \eqref{eq:I-PNP_2-sup} and \eqref{eq:I-PNP_2-Gamma_alpha},
  \begin{align*}
    &\E \sup_{t \in [0,T]} \left| \int_0^t (t-s)^{\alpha -1}\Pi_1 \Smu(t-s)\begin{pmatrix}  0 \\ (I-P_{N_\mu})  \end{pmatrix}(I-P_{N_\mu}) \Pi_2 \Gamma_\alpha^\mu(s)ds \right|_H^p\\
    & \leq C \left(\int_0^T s^{(\alpha-1)p/(p-1)} \left\| \Pi_1 \Smu(s)\begin{pmatrix}  0 \\ I - P_{N_\mu}  \end{pmatrix} \right\|_{\mathscr{L}(H^{-1},H)}^{p/(p-1)}ds \right)^{p-1} \\
    &\hspace{1cm} \times\int_0^T |(I -  P_{N_\mu}) \Pi_2 \Gamma_\alpha^\mu(s)|_{H^{-1}}^pds\\
    &\leq C \mu^{p/2}\E \int_0^T |(I-P_{N_\mu})\Pi_2\Gamma_\alpha^\mu(t)|_{H^{-1}}^p dt\\
    &\leq C \mu^{\gamma/2} \E \sup_{t \in [0,T]} \|\Phi(t)\|_{\mathscr{L}(L^\infty(D),H)}^p dt.
  \end{align*}

  Therefore the result follows.
\end{proof}

\begin{theorem} \label{thm:stoch-conv-sup-norm-Pi-2}
  Let $\Gamma^\mu(t)$ be given by \eqref{eq:Gamma-def}. For any $p \geq 1$ and $T\geq 0$, there exists a constant $C= C(p,T,\mu)$ such that
  \begin{equation}
    \E \sup_{t \in [0,T]} |\Gamma^\mu(t)|_\H^2 \leq C(T,p,\mu) \E \int_0^T \sup_{s \in [0,t]} \|\Phi(s)\|_{\mathscr{L}(L^\infty(D),H)}^p dt.
  \end{equation}
\end{theorem}
\begin{proof}
  The proof is similar to the proof of Theorem \ref{thm:stoch-conv-sup-norm}, but it is less complicated because the constant is allowed to depend on $\mu$. The main difference is that we use Lemma \ref{lem:Smu-pair-bound} instead of Lemmas \ref{lem:Pi1-Smu-I-0-op-norm}--\ref{lem:Smu-I-PN-op-norm} in the stochastic convolution argument. We omit further details.
\end{proof}

\section{Well-posedness of the stochastic wave equation -- Proof of Theorem \ref{thm:wave-eq-well-posed}} \label{S:well-posed}
Let $\mu>0$. We show that for any $(u_0,v_0) \in \H$ there is a unique solution mild $z^\mu \in C([0,T]:H)$ solving
\begin{align}  \label{eq:wave-eq-mild}
  z^\mu(t) = &S_\mu(t) \begin{pmatrix} u_0\\v_0 \end{pmatrix}
  + \int_0^t  S_\mu(t-s) \mathcal{I}_\mu B(s,\Pi_1z^\mu(s))ds\nonumber\\
  &+ \int_0^t  S_\mu(t-s) \mathcal{I}_\mu G(s,\Pi_1z^\mu(s))Qdw(s).
\end{align}
We prove well-posedness with the contraction mapping principle. Let $\mathscr{K}^\mu: L^p(\Omega:C([0,T]:\H)) \to L^p(\Omega:C([0,T]:\H))$ by
\begin{align} \label{eq:K-contraction}
  \mathscr{K}^\mu(\varphi)(t) = & S_\mu(t) \begin{pmatrix} u_0\\v_0 \end{pmatrix}
  + \int_0^t S_\mu(t-s) \mathcal{I}_\mu B(s,\Pi_1\varphi(s))ds\nonumber\\
  &+ \int_0^t  S_\mu(t-s) \mathcal{I}_\mu G(s,\Pi_1\varphi(s))Qdw(s).
\end{align}
Well-posedness follows from proving that there exists a unique fixed point for $\mathscr{K}^\mu$.

  For any $\varphi_1, \varphi_2 \in L^p(\Omega:C([0,T]:\H))$,
  \begin{align*}
    \E &\sup_{t \in [0,T]}|K^\mu(\varphi_1) - K^\mu(\varphi_2)|^p_\H \\
    \leq
    &C\E\sup_{t \in [0,T]}\left|\int_0^t S_\mu(t-s) \mathcal{I}_\mu (B(s,\Pi_1\varphi_1(s)) - B(s,\Pi_1\varphi_2(s)))ds \right|_\H^p\\
    &+C \E \sup_{t \in [0,T]} \left|\int_0^t S_\mu(t-s) \mathcal{I}_\mu (G(s,\Pi_1\varphi_1(s)) - G(s,\Pi_1 \varphi_2(s)))Qdw(s) \right|_\H^p.
  \end{align*}
  By Lemma \ref{lem:Smu-pair-bound},  $\sup_{t \geq 0} \|\Smu(t)\|_{\mathscr{L}(\H)} \leq \mu^{-1/2}$. By the Lipschitz continuity of $B$ (Assumption \ref{assum:bg}), for any $t \in [0,T]$,
  \begin{align*}
    &\left|\int_0^t  S_\mu(t-s) \mathcal{I}_\mu(B(\Pi_1\varphi_1(s))-B(\Pi_1\varphi_2(s))) ds  \right|_\H\\
    &\leq \mu^{-1/2}\int_0^t |B(s,\Pi_1\varphi_1(s)) - B(s,\Pi_1\varphi_2(s))|_H ds \\
    &\leq C \mu^{-1/2} \int_0^t |\Pi_1\varphi_1(s) - \Pi_1\varphi_2(s)|_Hds.
  \end{align*}
  For the stochastic term, Theorem \ref{thm:stoch-conv-sup-norm-Pi-2} and \eqref{eq:G-Lip} guarantee that
  \begin{align*}
    &\E \sup_{t \in [0,T]} \left|\int_0^t S_\mu(t-s) \mathcal{I}_\mu (G(s,\varphi_1(s)) - G(s,\varphi_2(s)))Qdw(s) \right|_H^p \\
    &\leq C(p,T,\mu) \E \int_0^T\sup_{s \in [0,t] }\|G(t,\Pi_1\varphi_1(s)) - G(t,\Pi_1\varphi_2(s))\|_{\mathscr{L}(L^\infty(D),H)}^p dt\\
    &\leq C(p,T,\mu) \E \int_0^T \sup_{s \in [0,t]}|\Pi_1\varphi_1(s) - \Pi_1\varphi_2(s)|_H^p dt.
  \end{align*}
  It follows from these two estimates that
  \[\E \sup_{t \in [0,T]}|K^\mu(\varphi_1) - K^\mu(\varphi_2)|^p_\H \\
    \leq C(T,p,\mu)\E \int_0^T \sup_{s \in [0,t]}|\Pi_1\varphi_1(t) - \Pi_1\varphi_2(t)|_H^p dt.\]
   Therefore, for small enough $T_0>0$, $\mathscr{K}^\mu$ is a contraction on $L^p(\Omega:C([0,T_0]:\H))$. We can use standard methods to append solutions in the intervals $[0,T_0]$,$ [T_0,2T_0]$,\\$ [2T_0,3T_0],$... to get a unique solution to \eqref{eq:wave-eq-mild} in $L^p(\Omega:C([0,T]:H))$ for any $T>0$.
%\end{proof}

\section{Convergence -- Proof of Theorem \ref{thm:convergence-u}} \label{S:convergence}
Before proving Theorem \ref{thm:convergence-u}, we state two auxilliary results about the convergence of the stochastic convolutions and Lebesgue integral convolutions with the wave and heat semigroups.
We state a result about the convergence of the stochastic convolutions $\Gamma^\mu$ to $\Gamma$ where
$\Gamma^\mu$ defined in \eqref{eq:Gamma-def} converge to $\Gamma$ defined in \eqref{eq:Gamma-heat-def}.
\begin{theorem} \label{thm:convergence-Gammas}
  Let $T>0$, let $\alpha \in (0,1/2)$ satisfy $-2\alpha - \frac{d(q-2)}{2q}>-1$ and let $p > \frac{1}{\alpha}$. For any self-adjoint, progressively measurable $\Phi \in L^p(\Omega:L^\infty([0,T]:\mathscr{L}(L^\infty(D),H)))$ let $\Gamma^\mu$ and $\Gamma$ be \eqref{eq:Gamma-def} and \eqref{eq:Gamma-heat-def} respectively. Then
  \begin{equation}
    \lim_{\mu \to 0} \E |\Pi_1\Gamma^\mu - \Gamma|_{C([0,T]:H)}^p = 0.
  \end{equation}
\end{theorem}
  Theorem \ref{thm:convergence-Gammas} is really the most technical piece of this paper. We will delay its proof to subsection \ref{S:conv-proof}.
We will need a similar result about the Lebesgue integrals.

\begin{theorem} \label{thm:convergence-Lebesgue}
  For any $T>0$ and $\varphi \in L^\infty([0,T]:H)$,
  \begin{equation} \label{eq:conv-Lebesgue}
    \lim_{\mu \to 0}\sup_{t \in [0,T]} \left|\int_0^t(S(t-s) - \Pi_1 \Smu(t-s) \I_\mu) \varphi(s) ds\right|_H=0.
  \end{equation}
\end{theorem}
The proof is in subsection \ref{S:conv-proof-Lebesgue}.

We now prove the main convergence result via Theorems \ref{thm:convergence-Gammas} and \ref{thm:convergence-Lebesgue}.
\begin{proof}[Proof of Theorem \ref{thm:convergence-u}]
  We decompose the difference between the mild solutions \eqref{eq:wave-mild} and \eqref{eq:heat-mild} into the following pieces
  \begin{align} \label{eq:Js}
    &u(t) - u^\mu(t)  = (S(t)u_0 - \Pi_1 \Smu(t) (u_0,v_0))\nonumber\\
    &+ \int_0^t (S(t-s) - \Pi_1 \Smu(t-s) \I_\mu) B(s,u(s))ds\nonumber\\
    &+ \int_0^t \Pi_1 \Smu(t-s)\I_\mu(B(s,u(s)) - B(s,u^\mu(s)))ds\nonumber\\
    &+\left[ \int_0^t S(t-s)G(s,u(s))Qdw(s) - \int_0^t \Pi_1\Smu(t-s) \I_\mu G(s,u(s))Q dw(s)\right]\nonumber\\
    &+ \int_0^t \Pi_1 \Smu(t-s) \I_\mu (G(s,u(s)) - G(s,u^\mu(s)))Qdw(s)\nonumber\\
    &=:\sum_{k=1}^5 J^\mu_k(t).
  \end{align}

  Letting $u_k = \left<u_0,e_k\right>_H$  it follows from \eqref{eq:semigroup-decomp} that
  \begin{align*}
    &\sup_{t \in [0,T]}|S(t) u_0 - \Pi_1\Smu(t)(u_0,0)|_H^2
    = \sum_{k=1}^\infty u_k^2 \sup_{t \in [0,T]}( e^{-\alpha_k t} - f_k^\mu(t;1,0))^2
  \end{align*}
  The above expression converges to zero by the dominated convergence theorem and \eqref{eq:f-conv-u}.
  Similarly, letting $v_k = \left<v_0,e_k\right>_H$, and $N_\mu = \max\{k \in \mathbb{N}: 1-4\mu\alpha_k\geq0\}$ it follows from \eqref{eq:f-bound-alpha} and \eqref{eq:f-bound-mu} that
  \begin{align*}
    \left|\Pi_1 \Smu(t)(0,v_0) \right|_H^2 = \sum_{k=1}^\infty v_k^2 \left[f^\mu_k(t:0,1)\right]^2
    \leq \sum_{k=1}^{N_\mu} v_k^2 16\mu^2 + \sum_{k=N_\mu + 1}^\infty \frac{4\mu v_k^2}{\alpha_k}
  \end{align*}
  If $k \leq N_\mu$, then $1-4\mu\alpha_k \geq 0$. In particular, $\mu \leq \frac{1}{4\alpha_k}$ and $\mu^2 \leq \frac{\mu}{4\alpha_k}$. Applying this bound to the first sum in the above display, it follows that
  \begin{align*}
    \left|\Pi_1 \Smu(t)(0,v_0) \right|_H^2 \leq 4\mu \sum_{k=1}^\infty \frac{v_k^2}{\alpha_k} \leq 4\mu |v|^2_{H^{-1}}.
  \end{align*}
  These calculations show that
  \begin{align} \label{eq:J_1}
    &\lim_{\mu \to 0} \sup_{t \in [0,T]} |J_1^\mu(t)|_H \nonumber\\
    &\leq \lim_{\mu \to 0} \sup_{t \in [0,T]} \left(|S(t)u_0 - \Pi_1 \Smu(t)(u_0,0)|_H + |\Pi_1 \Smu(t)(0,v_0)|_H \right) = 0.
  \end{align}

  By Theorem \ref{thm:heat-well-posed}, the unique solution to \eqref{eq:heat-mild} is in $ L^p(\Omega:C([0,T]:H))$. By the linear growth of $B$ (see \eqref{eq:linear-growth}), $B(\cdot, u(\cdot)) \in L^p(\Omega:C([0,T]:H)$ as well. It follows from Theorem \ref{thm:convergence-Lebesgue} and the dominated convergence theorem that
  \begin{equation} \label{eq:J_2}
    \lim_{ \mu \to 0} \sup_{t \in [0,T]}\E|J_2(t)|_H^p = 0.
  \end{equation}

  By the Lipschitz continuity of $B$ \eqref{eq:Lipschitz}, there exists a constant $C>0$ such that for all $s \in [0,T]$, $|B(s,u(s)) - B(s,u^\mu(s))|_H \leq C |u(s) - u^\mu(s)|_H$. By Lemma \ref{lem:pi1-Smu-I-mu-bounded} and a H\"older inequality,
  \begin{equation} \label{eq:J_3}
    \sup_{t \in [0,T]} \E | J_3(t)|^p \leq C T^{p-1} \E \int_0^T \sup_{s \in [0,t]}|u(s) - u^\mu(s)|^p dt.
  \end{equation}

  From the linear growth of $G$ \eqref{eq:linear-growth} and the fact that $u \in L^p(\Omega: C([0,T]:H))$, it follows that $G(\cdot, u(\cdot)) \in L^p(\Omega: L^\infty([0,T]:\mathscr{L}(L^\infty(D),H)))$. Theorem \ref{thm:convergence-Gammas} implies that
  \begin{equation} \label{eq:J_4}
    \lim_{\mu \to 0} \sup_{t \in [0,T]} |J_4(t)|_H^p =0.
  \end{equation}

  By Theorem \ref{thm:stoch-conv-sup-norm}
  \[\E \sup_{t \in [0,T]} |J_5(t)|_H^p \leq C \E \int_0^T \sup_{s \in [0,t]} \|G(s,u(s)) - G(s,u^\mu(s))\|_{\mathscr{L}(L^\infty(D),H)}^p dt.\]
  By the Lipschitz continuity of $G$ \eqref{eq:G-Lip}, there exists a constant independent of $s$ and $\mu$ such that $\|G(s,u(s)) - G(s,u^\mu(s))\|_{\mathscr{L}(L^\infty(D),H)} \leq C|u(s) - u^\mu(s)|_H$. It follows that
  \begin{equation} \label{eq:J_5}
    \E \sup_{t \in [0,T]} |J_5(t)|_H^p \leq C(T) \E \int_0^T \sup_{s \in [0,t]}|u(s) - u^\mu(s)|_H^p dt.
  \end{equation}

  It now follows from \eqref{eq:Js}, \eqref{eq:J_3}, and \eqref{eq:J_5}, that there exists an increasing $C(T)>0$ such that for any $T>0$
  \begin{align*}
    \E \sup_{t \in [0,T]}|u^\mu(t)-u(t)|_H^p \leq C(T) \Bigg(\sup_{t \in [0,T]}|J_1(t)|_H^p + \sup_{t \in [0,T]} \E|J_2(t)|_H^p& \\
    + \sup_{t \in [0,T]}\E |J_4(t)|_H^p +  \E \int_0^T\sup_{s \in [0,t]}|u(s) - u^\mu(s)|_H^p dt &\Bigg).
  \end{align*}
  By Gr\"onwall's inequality, for any $T>0$,
  \begin{align*}
    &\E \sup_{t \in [0,T]}|u^\mu(t)-u(t)|_H^p \\
    &\leq C(T)e^{TC(T)} \left(\sup_{t \in [0,T]} |J_1(t)|^p + \sup_{t \in [0,T]} \E|J_2(t)|_H^p + \sup_{t \in [0,T]}\E |J_4(t)|_H^p \right).
  \end{align*}
  Finally, we conclude that the above display converges to zero due to \eqref{eq:J_1}, \eqref{eq:J_2}, and \eqref{eq:J_4}.
\end{proof}

\subsection{Proof of Theorem \ref{thm:convergence-Gammas}} \label{S:conv-proof}
\begin{lemma} \label{lem:Gamma-alpha-mu-to-Gamma-alpha}
  Let $\alpha>0$, $p>\frac{1}{\alpha}$ and $\Phi \in L^p(\Omega : L^\infty([0,T]:\mathscr{L}(L^\infty(D),H)))$ satisfy the assumptions of Theorem \ref{thm:convergence-Gammas}. Let $\Gamma_\alpha^\mu$ be given by \eqref{eq:Gamma-mu-alpha} and $\Gamma_\alpha$ be given by \eqref{eq:Gamma-alpha}. For any $t>0$,
  \[\lim_{\mu \to 0} \E |\Pi_1\Gamma_\alpha^\mu(t) - \Gamma_\alpha(t)|_H^p =0.\]
\end{lemma}

\begin{proof}
  The scalar quadratic variation of $\Pi_1\Gamma_\alpha^\mu(t) - \Gamma_\alpha(t)$ is
  \[\Lambda(t) = \sum_{j=1}^\infty \int_0^t (t-s)^{-2\alpha} |(\Pi_1 S_\mu(t-s)\mathcal{I}_\mu - S(t-s))\Phi(s)Q e_j|_H^2ds.\]
  Writing this expression as a double sum and using the fact that $e_k$ are eigenfunctions for $S(t)$, $\Pi_1\Smu(t)\I_\mu$ and $Q$,
  \[\Lambda(t) = \sum_{k=1}^\infty \sum_{j=1}^\infty \int_0^t (t-s)^{-2\alpha} \lambda_j^2 |(f^\mu_k)(t:0,1/\mu) - e^{-\alpha_k t}|^2\left<\Phi(s)e_j, e_k\right>_H^2ds.\]
  For fixed $k,j \in \mathbb{N}$ and $s \in[0,t]$, this integrand is dominated by,
  \begin{align*}
     2(t-s)^{-2\alpha} \lambda_j^2\left( |(f^\mu_k)(t;0,1/\mu)|^2 + e^{-2\alpha_k t}\right) \left<\Phi(s)e_j, e_k\right>_H^2\\
  \end{align*}
  which is integrable by the arguments of Lemma \ref{lem:Pi_1-stoch-conv} and \cite[Section 3]{c-2003}.
  By \eqref{eq:f-conv-v} and the dominated convergence theorem, $\Lambda(t) \to 0$. By the BDG inequality,
  \[\lim_{\mu \to 0} \E |\Pi_1\Gamma_\alpha^\mu(t) - \Gamma_\alpha(t)|_H^p =0.\]
\end{proof}

\begin{lemma} \label{lem:S-Smu-op-norm}
  For any $N \in \mathbb{N}$ and $t\geq0$,
  \[\lim_{\mu \to 0} \left\|\Pi_1 S_\mu(t) \begin{pmatrix}P_N \\ 0  \end{pmatrix} - S(t) P_N \right\|_{\mathscr{L}(H)} = 0.\]
\end{lemma}
\begin{proof}
  Notice that because these operators are diagonalized by the orthonormal basis $\{e_k\}$,
  \[\left\|\Pi_1 S_\mu(t) \begin{pmatrix}P_N \\ 0  \end{pmatrix} - S(t) P_N \right\|_{\mathscr{L}(H)}
  = \max_{k \leq N} |f_k^\mu(t;1,0) - e^{-\alpha_k t}|,\]
  and the above expression converges to zero by \eqref{eq:f-conv-u}. The limit will not be true without the projection onto a finite dimensional span.
\end{proof}

\begin{proof}[Proof of Theorem \ref{thm:convergence-Gammas}]
  By the factorization method of \cite[Chapter 5.3.1]{dpz},
  \[\Gamma(t) = \int_0^t (t-s)^{\alpha -1} S(t-s)\Gamma_\alpha(s)ds, \ \ \ \Gamma^\mu(t) = \int_0^t (t-s)^{\alpha-1} \Smu(t-s) \Gamma^\mu_\alpha(s)ds,\]
  where $\Gamma_\alpha$ and $\Gamma^\mu_\alpha$ are defined in \eqref{eq:Gamma-alpha} and \eqref{eq:Gamma-mu-alpha}.

  We split up the difference into five pieces. Let $N \in \mathbb{N}$ be chosen later. Let $N_\mu = \sup\{k \in \mathbb{N}: 1-4\mu\alpha_k \geq 0\}$.
  \begin{align} \label{eq:I-decomp}
    &\Gamma(t)- \Pi_1\Gamma^\mu(t) = \nonumber\\
    &+ \int_0^t (t-s)^{\alpha -1} \left(S(t-s)P_N - \Pi_1 S_\mu(t-s) \begin{pmatrix} P_N \\0 \end{pmatrix}\right) \Gamma_\alpha(s)ds\nonumber\\
    &+ \int_0^t (t-s)^{\alpha -1} \left(S(t-s)(I-P_N) - \Pi_1 S_\mu(t-s) \begin{pmatrix}I- P_N \\0 \end{pmatrix}\right) \Gamma_\alpha(s)ds\nonumber\\
    &+\int_0^t (t-s)^{\alpha-1} \Pi_1S_\mu(t-s)\begin{pmatrix}I \\ 0 \end{pmatrix} (\Gamma_\alpha(s) - \Pi_1 \Gamma^\mu_\alpha(s))ds\nonumber\\
    &-\int_0^t (t-s)^{\alpha-1} \Pi_1 S_\mu(t-s) \mathcal{I}_1 P_{N_\mu} \Pi_2 \Gamma^\mu_\alpha(s)\nonumber\\
    &- \int_0^t (t-s)^{\alpha -1} \Pi_1 S_\mu(t-s) \mathcal{I}_1( I - P_{N_\mu}) \Pi_2 \Gamma^\mu_\alpha(s)\nonumber\\
    &=: I^\mu_{1,N}(t) + I^\mu_{2,N}(t) + I^\mu_{3,N}(t) + I^\mu_{4,N}(t) + I^\mu_{5,N}(t).
  \end{align}
  We also denote $I^\mu_i(t) := I^\mu_{i,n}(t)$ for $i = 3,4,5$ because these terms are independent of the choice of $N$.

  By the H\"older inequality, for $p>\frac{1}{\alpha}$ and $N \in \mathbb{N}$,
  \begin{align*}
    \E&\sup_{t \in [0,T]}|I_{1,N}^\mu(t)|_H^p \\
    \leq &\left(\int_0^T s^{\frac{(\alpha-1)p}{p-1}} \left\|S(s)P_N - \Pi_1 S_\mu(s) \begin{pmatrix} P_N \\0 \end{pmatrix}\right\|_{\mathscr{L}(H)}^{\frac{p}{p-1}} ds\right)^{p-1}\\
    &\times \int_0^T \E |P_N \Gamma_\alpha(s)|_H^p ds.
  \end{align*}
  By Lemma \ref{lem:S-Smu-op-norm} and the dominated convergence theorem, for any fixed $N \in \mathbb{N}$,
  \[\lim_{\mu \to 0} \left(\int_0^T s^{\frac{(\alpha-1)p}{p-1}} \left\|S(s)P_N - \Pi_1 S_\mu(s) \begin{pmatrix} P_N \\0 \end{pmatrix}\right\|_{\mathscr{L}(H)}^{\frac{p}{p-1}}ds \right)^{p-1} = 0.\]
  The dominated convergence is valid by Lemma \ref{lem:S-Smu-op-norm}, the well-known fact that the heat equation semigroup is uniformly bounded, and the fact that $p> \frac{1}{\alpha}$ implies $\frac{(\alpha-1)(p-1)}{p}>-1$.

  Note that Lemma \ref{lem:heat-Gamma-alpha-bound} implies that
  $ \E|\Gamma_\alpha(t)|_H^p $ is bounded uniformly in $t \in [0,T]$. It follows that for any fixed $N \in \mathbb{N}$,
  \begin{equation} \label{eq:I_1}
    \lim_{\mu \to 0} \sup_{t \in [0,T]} |I^\mu_{1,N(t)}|_H^p = 0.
  \end{equation}

  Now we show that $I^\mu_{2,N}$ converges to $0$ as $N\to +\infty$ independently of $\mu>0$. By the H\"older inequality,
  \begin{align*}
    \E&\sup_{t \in [0,T]}|I_{2,N}^\mu(t)|_H^p \\
    \leq &\left(\int_0^T s^{\frac{(\alpha-1)p}{p-1}} \left\|S(s)(I-P_N) - \Pi_1 S_\mu(s) \begin{pmatrix} I- P_N \\0 \end{pmatrix}\right\|_{\mathscr{L}(H)}^{\frac{p}{p-1}} ds\right)^{p-1}\\
    &\times \int_0^T \E |(I-\Pi_N) \Gamma_\alpha(s)|_H^p ds.
  \end{align*}
  The first integral is uniformly bounded by Lemma \ref{lem:Pi1-Smu-I-0-op-norm} and the boundedness of the heat equation semigroup. Specifically, for any $N \in \mathbb{N}$ and $\mu \in (0,1)$,
  \begin{align*}
    &\left\|S(s)(I-P_N) - \Pi_1\Smu(s)\begin{pmatrix} I- P_N \\0 \end{pmatrix}\right\|_{\mathscr{L}(H)} \\
    &\leq \|S(s)\|_{\mathscr{L}(H)} + \left\| \Pi_1\Smu(s)\begin{pmatrix} I \\0 \end{pmatrix}\right\|_{\mathscr{L}(H)} \leq 2.
  \end{align*}
  For any fixed $s\in [0,T]$, $\E|(I-P_N)\Gamma_\alpha(s)|_H^p$ converges to $0$ as $N \to +\infty$ by Lemma \ref{lem:Gamma-tails}. Therefore,
  \begin{equation} \label{eq:I_2}
    \lim_{N \to +\infty} \sup_{\mu \in (0,1)} \E\sup_{t \in [0,T]}|I_{2,N}^\mu(t)|_H^p=0.
  \end{equation}

  For $I^\mu_3$, we notice that
  \begin{align*}
    &\E\sup_{t \in [0,T]}|I_3^\mu(t)|_H^p\\
    &\leq \left(\int_0^T s^{\frac{(\alpha-1)p}{p-1}}  \left\|\Pi_1 S_\mu(s) \begin{pmatrix} I\\0 \end{pmatrix} \right\|_{\mathscr{L}(H)}^{\frac{p}{p-1}} ds\right) \int_0^T \E|\Gamma_\alpha(s) - \Pi_1\Gamma^\mu_\alpha(s)|_H^pds.
  \end{align*}
  Lemma \ref{lem:Pi1-Smu-I-0-op-norm} guarantees that the first integral is uniformly bounded. Lemma \ref{lem:Gamma-alpha-mu-to-Gamma-alpha} and the dominated convergence theorem guarantees that
  \begin{equation} \label{eq:I_3}
    \lim_{\mu \to 0} \E \sup_{t \in [0,T]}|I_3^\mu(t)|_H^p = 0.
  \end{equation}
  The dominated convergence is valid due to Lemma \ref{lem:Pi_1-stoch-conv}.

  For $I^\mu_4$,
  \begin{align*}
    &\E \sup_{t \in [0,T]} |I_4^\mu(t)|_H^p\\
    &\leq \left(\int_0^T s^{\frac{(\alpha-1)p}{p-1}} \left\|\Pi_1 S_\mu(s) \mathcal{I}_\mu P_{N_\mu} \right\|_{\mathscr{L}(H)}^{\frac{p}{p-1}} ds\right)^{p-1}\int_0^T \E|\mu \Pi_2\Gamma^\mu_\alpha(s)|_H^pds.
  \end{align*}
  The first integral is bounded by Lemma \ref{lem:Pi1Smu-op-norm}. The second integral goes to zero as $\mu$ goes to zero by \eqref{eq:PNPi_2-Gamma_alpha}, \eqref{eq:Pi_2-stoch-conv-to-0}, and the dominated convergence theorem. Therefore,
  \begin{equation} \label{eq:I_4}
    \lim_{\mu \to 0} \E \sup_{t \in [0,T]} |I_4^\mu(t)|_H^p = 0.
  \end{equation}

  Finally,
  \begin{align*}
    \E \sup_{t \in [0,T]} |I^\mu_5(t)|_H^p \leq &\left(\int_0^T s^{\frac{(\alpha-1)p}{p-1}} \left\|\Pi_1 S_\mu(s) \mathcal{I}_1 (I-P_{N_\mu}) \right\|_{\mathscr{L}(H^{-1},H)} ds\right)^{p-1}\\
    &\times\int_0^T\E|(I-P_{N_\mu})\Pi_2\Gamma^\mu_\alpha(s)|_{H^{-1}}^pds.
  \end{align*}
  By Lemma \ref{lem:Smu-I-PN-op-norm}, there exists $C>0$ independent of $\mu$ such that
  \[\left(\int_0^T s^{\frac{(\alpha-1)p}{p-1}} \left\|\Pi_1 S_\mu(s) \mathcal{I}_1 (I-P_{N_\mu}) \right\|_{\mathscr{L}(H^{-1},H)}^{\frac{p}{p-1}}ds \right)^{p-1} \leq C\mu^{\frac{p}{2}}.\]
  By \eqref{eq:I-PNP_2-Gamma_alpha},
  \[\int_0^T\E|(I-P_{N_\mu})\Pi_2 \Gamma^\mu_\alpha(s)|_{H^{-1}}^p ds \leq \frac{CT}{\mu^{(p-\zeta)/2}}\E\sup_{s \in [0,T]} \|\Phi(s)\|_{\mathscr{L}(H)}^p.\]
  Therefore,
  \begin{equation} \label{eq:I_5}
    \lim_{\mu \to 0} \E\sup_{t \in [0,T]} |I^\mu_5(t)|_H^p=0.
  \end{equation}

  We can now complete the proof. Pick any arbitrary $\eta>0$. There exists a constant $C>0$ such that by \eqref{eq:I-decomp},
  \begin{align*}
    \E \sup_{t \in [0,T]}|\Gamma(t) - \Gamma^\mu(t)|_H^p \leq C \sum_{i=1}^5\sup_{t \in [0,T]}|I_{i,N}^\mu(t)|_H^p.
  \end{align*}
  Choose $N$ large enough so that by \eqref{eq:I_2}, $\E\sup_{t \in [0,T]}|I^\mu_{2,N}(t)|_H^p<\frac{\eta}{5C}$. Then choose $\mu_0>0$ small enough so that for any $\mu \in (0,\mu_0)$, \eqref{eq:I_1}, \eqref{eq:I_3}, \eqref{eq:I_4}, and \eqref{eq:I_5} guarantee that
  $\E\sup_{t \in [0,T]}|I^\mu_{i,N}(t)|_H^p<\frac{\eta}{5C}$ for $i=1,3,4,5$. Then for $\mu \in (0,\mu_0)$,
  \[\E\sup_{t \in [0,T]}|\Gamma(t) - \Gamma^\mu(t)|_H^p < \eta.\]
  The result follows because $\eta>0$ was arbitrary.
\end{proof}

\subsection{Proof of Theorem \ref{thm:convergence-Lebesgue}} \label{S:conv-proof-Lebesgue}
Let $P_N$ be the projection onto the finite dimensional span $\{e_k\}_{k=1}^N$. The following lemma is a consequence of  \eqref{eq:f-conv-v}.
\begin{lemma} \label{lem:Smu-conv-op-norm}
  For any $0<t_0<T$ and $N \in \mathbb N$,
  \begin{equation} \label{eq:op-norm-conv}
     \lim_{ \mu \to 0} \sup_{t \in [t_0,T]} \left\|(S(t) - \Pi_1 \Smu(t) \I_\mu )P_N \right\|_{\mathscr{L}(H)} = 0.
  \end{equation}
\end{lemma}
  \begin{proof}
    Because for any fixed $t>0$, the operators $S(t)$ and $\Pi_1\Smu(t) \I_\mu$ are both diagonalized by the orthonormal basis $\{e_k\}$,
    \[\left\|(S(t) - \Pi_1 \Smu(t) \I_\mu )P_N \right\|_{\mathscr{L}(H)} = \max_{ k \in \{1,...,N\}} |f^\mu_k(t;0,1/\mu) - e^{-\alpha_k t}|\]
    where $f^\mu_k(t;0,1/\mu)$ solves \eqref{eq:f_k-def}. The result follows by \eqref{eq:f-conv-v} and the fact that we are only working with a finite number of modes at a time.
  \end{proof}

\begin{proof}[Proof of Theorem \ref{thm:convergence-Lebesgue}]
Let $T>0$ and $\varphi \in L^\infty([0,T]:H)$. For any $N \in \mathbb{N}$,
\begin{align} \label{eq:Leb-conv-ineq}
  \Bigg|\int_0^t &(S(t-s) - \Pi_1 \Smu(t-s)\I_\mu) \varphi(s)ds \Bigg|_H\nonumber\\
  \leq& \int_0^t |(S(t-s) - \Pi_1 \Smu(t-s) \I_\mu) P_N \varphi(s)|_Hds \nonumber\\
  &+ \int_0^t |(S(t-s) - \Pi_1 \Smu(t-s)\I_\mu)( I - P_N) \varphi(s)|_Hds\nonumber\\
  \leq&  \left(\int_0^t \|(S(t-s) - \Pi_1 \Smu(t-s) \I_\mu) P_N \|_{\mathscr{L}(H)}ds \right)|\varphi|_{L^\infty([0,T]:H)} \nonumber \\
  &+ 5 \int_0^t |(I - P_N)\varphi(s)|_Hds.
\end{align}
The last inequality is due to the fact that by Lemma \ref{lem:pi1-Smu-I-mu-bounded} for any $t\geq 0$,
\[\|S(t) - \Pi_1\Smu(t)\I_\mu\|_{\mathscr{L}(H)} \leq \|S(t)\|_{\mathscr{L}(H)} + \|\Pi_1 \Smu(t) \I_\mu\|_{\mathscr{L}(H)} \leq 5.\]
It follows from \eqref{eq:Leb-conv-ineq} that
\begin{align}
  \sup_{t \in [0,T]} &\left|\int_0^t (S(t-s) - \Pi_1 \Smu(t-s)\I_\mu) \varphi(s)ds \right|_H \nonumber \\
  \leq& \left(\int_0^T \|(S(t-s) - \Pi_1 \Smu(t-s) I_\mu) P_N \|_{\mathscr{L}(H)}ds \right)|\varphi|_{L^\infty([0,T]:H)} \nonumber\\
  &+ 5 \int_0^T |(I - P_N)\varphi(s)|_Hds.
\end{align}
By Lemma \ref{lem:Smu-conv-op-norm} and the dominated convergence theorem,
\[\lim_{\mu \to 0} \sup_{t \in [0,T]} \left|\int_0^t (S(t-s) - \Pi_1 \Smu(t-s)\I_\mu) \varphi(s)ds \right|_H \leq 5 \int_0^T |(I - P_N)\varphi(s)|_Hds.\]
Finally, we recall that $N \in \mathbb{N}$ was arbitrary and that the dominated convergence theorem guarantees that the limit of the right-hand side as $N \to +\infty$ is 0.

\end{proof}

\bibliography{2017}
\end{document}